\documentclass[a4paper,11pt,reqno]{amsart}

\usepackage{amsmath,multicol} %a4wide 
\usepackage[all]{xy}

\newtheorem{theorem}{\bf Theorem}[subsection]
\newtheorem{prop}[theorem]{\bf Proposition}
\newtheorem{cor}[theorem]{\bf Corollary}
\newtheorem{lemma}[theorem]{\bf Lemma}

\newtheorem{definition}[theorem]{\bf Definition}

\theoremstyle{remark}

\theoremstyle{remark}
\newtheorem{rem}[theorem]{\bf Remark}

 \numberwithin{equation}{subsection}

\newcommand{\ad}{\operatorname{ad}}
\newcommand{\Ad}{\operatorname{Ad}}
\newcommand{\rank}{\operatorname{rank}}

\def\go{\mathfrak}
\def\bb{\mathbb}
\def\cal{\mathcal}

 \def\adots{\mathinner{\mkern2mu\raise1pt\hbox{.}
\mkern3mu\raise4pt\hbox{.}\mkern1mu\raise7pt\hbox{.}}}

 \title[Howe Correspondence I]{Invariant differential operators and  an infinite dimensional Howe-type correspondence.\\
 {\rm \small{Part I: Structure of the associated algebras of differential operators}}}

\author{Hubert Rubenthaler}

\address
{Hubert Rubenthaler\\ Institut de Recherche Math\'ematique Avanc\'ee\\
Universit\'e Louis Pasteur et CNRS\\
7 rue Ren\'e Descartes\\
67084 Strasbourg Cedex\\ France\\
E-mail: {\tt rubenth@math.u-strasbg.fr}}

\begin{document} 

 \maketitle

\begin{abstract}
If $Q$ is a non degenerate quadratic form on ${\bb C}^n$, it is well known that the differential operators $X=Q(x)$, $Y=Q(\partial)$, and $H=E+\frac{n}{2}$, where $E$ is the Euler operator, generate a Lie algebra isomorphic to ${\go sl}_{2}$. Therefore the associative algebra they generate is a quotient of the universal enveloping algebra ${\cal U}({\go sl}_{2})$. This fact is in some sense the foundation of the metaplectic representation. The present  paper   is devoted to the study of the case where $Q(x)$ is replaced by $\Delta_{0}(x)$, where $\Delta_{0}(x)$ is the relative invariant of a prehomogeneous vector space of  commutative parabolic   type ($ {\go g},V $), or equivalently where $\Delta_{0}$ is the "determinant" function of a simple Jordan algebra $V$ over ${\bb C}$. In this Part I we show several structure results for  the associative algebra generated by $X=\Delta_{0}(x)$, $Y=\Delta_{0}(\partial)$. Our main result shows that if we consider this algebra as an algebra over a certain commutative ring ${\bf A}$ of invariant differential operators  it is isomorphic to the quotient of   what we call a generalized Smith algebra $S(f, {\bf A}, n)$  where  $f\in {\bf A}[t]$. The Smith algebras (over ${\bb C}$) were introduced by P. Smith  as "natural" generalizations of ${\cal U}({\go sl}_{2})$. In the forthcoming Part II we will consider    the Lie algebra ${\cal L}$ generated by $X, Y$ and ${\go gl}(V)$,     the Lie algebra ${\cal A}$ generated by $X$ and $Y$, and put  ${\cal B}= {\go g}$  (where ${\go g}$ is the structure algebra). Then ${\cal A}$ and ${\cal B}$ are commuting subalgebras of ${\cal L}$. Moreover the restriction of the natural representation of ${\cal L}$ on polynomials on $V$ to ${\cal A}\times {\cal B}$ gives rise to a correspendence between some highest weight modules of ${\cal A}$ and the "harmonic" representation of ${\cal B}$, which generalizes the Howe correspondence between highest weight modules of ${\go sl}_{2}$ and ordinary spherical harmonics. The Lie algebras  ${\cal L}$ and ${\cal A}$ are infinite-dimensional except if  $\Delta_{0}$ is a quadratic form, and in this case ${\cal L}$ is the usual symplectic algebra, ${\cal A}={\go sl}_{2}$ and the above mentioned representation is the infinitesimal metaplectic representation.
 
 \end{abstract}
 
\maketitle

\medskip\medskip\medskip\medskip\medskip\medskip
%%%%%%%%%%%%%%%%%%%%%%%%%%%%%%%%%%%%%%%%%%%%%%%%%%%%%%%%%%%%%%%%%%%%%%%%%%%%%%%%%%%%%%%%

 \section{Introduction}
 \vskip 10pt
%%%%%%%%%%%%%%%%%%%%%%%%%%%%%%%%%%%%%%%%%%%
 \subsection{}
  Let  $Sp(n,{\bb R})$ be the real symplectic group on rank $n$ and let $\widetilde{Sp}(n,{\bb R})$ be  its two fold covering group,   the so-called metaplectic group. There is a unitary representation  of $\widetilde{Sp}(n,{\bb R})$, constructed by Shale [Sh] and Weil [We], which we will call the {\it metaplectic} representation (and denote by $\pi$),  that is of considerable interest in representation theory.
 In this paper we will propose an infinitesimal generalization of the  metaplectic  representation, which is different from the   {\it minimal representations}. Let us recall some basic facts about this representation and the infinitesimal Howe correspondence between harmonic representations of ${\go o}(n)$ and some lowest-weight modules of ${\go sl}(2)$.  
 
Let $ \widetilde{U }$ be the two fold covering group of the unitary group $U $ in $n$ variables. The group $U $ (respectively $\widetilde{U }$ ) is the maximal compact subgroup of $Sp(n,{\bb R})$ (respectively  $\widetilde{Sp}(n,{\bb R}) $ ). The metaplectic representation $\pi$ of $\widetilde{Sp}(n, {\bb R}) $ is usually realized in $L^2({\bb R})$. However there exists a realization of the corresponding (${\go sp}(n,{\bb C}) , \widetilde{U }$)-module of $\pi$, called the Fock model ([H-1]), where the space of $\widetilde{ U}$-finite vectors is the space ${\bb C}[{\bb C}^n]$ of polynomials in $n$ variables.

In order to describe  explicitely the corresponding infinitesimal representation $\pi^{\infty}$ of the complexified Lie algebra $ {\go sp}(n,{\bb C})$, we need first to remark that   $ {\go sp}(n,{\bb C})$ is 3-graded:
$$ {\go sp}(n,{\bb C})=V^-\oplus {\go g}\oplus V^+$$
where $V^-\simeq V^+\simeq  \rm{Sym}_{\bb C}= $ the space of symmetric $n \times n$ complex matrices and where ${\go g} \simeq {\go gl}(n,{\bb C})$. More explicitely 
$${\go g}=   \{ \begin{pmatrix}
A&0 \\
0&-^t A
\end{pmatrix},  A\in {\go gl}(n,{\bb C})\}
 $$
 and 
 $$V^+=\{   \begin{pmatrix}
0&X \\
0&0
\end{pmatrix} , X\in   \rm{Sym}_{\bb C}\}\text {\,  and \, } V^-=\{   \begin{pmatrix}
0&0 \\
Y&0
\end{pmatrix} , Y\in   \rm{Sym}_{\bb C}\}.$$
Let   $Q_{X}$ denote the quadratic form associated to the symmetric matrix $X$. In other words $Q_{X}(x)= {^{t}x} Xx $, where $x$ is a column vector in  ${\bb C}^n$. Let us also denote by $Q_{X}(\partial)$ the differential operator with constant coefficients obtained by replacing $x_{i}$ by $ \frac {\partial} {\partial x_{i}  }$. Then in the Fock model we have, for all $P\in {\bb C}[{\bb C}^n]$ (see [H-1]):
 
\begin{equation} \begin{array}{l}  \forall X\in V^+, \pi^{\infty}(X)P(x) =Q_{X}(x)P(x) \\
\\
      \forall Y\in V^-, \pi^{\infty}(Y)P(x) =Q_{Y}(\partial)P(x) \\
      \\
    \forall A\in { \go g}, \pi^{\infty}(A)P(x) =\frac{Tr(A)}{2}(x)+(A.P)(x)= \frac{Tr(A)}{2}(x)+P'(x)(Ax)     
    \end{array} \end{equation}
    
    On the other hand, if the quadratic form $Q_{X}$ is non degenerate (i.e if $\det(X)\neq 0$), then we consider the embedding of ${\go sl}_{2}$ into $sp(n,{\bb C})$ defined by  
    
    \begin{equation}\begin{pmatrix}
    a&b\\
    c&-a\\
    \end{pmatrix}\hookrightarrow \begin{pmatrix}
    a\text{Id}&bX\\
    cX^{-1}&-a\text{Id}\\
    \end{pmatrix}\end{equation}    
    From $(1.1)$ and $(1.2)$ we obtain the following ${\go sl}_{2}$-triple of differential operators:
    \begin{equation}Y=Q_{X^{-1}}(\partial), H=E+\frac{n}{2},  X=Q_{X}(x)  \end{equation}
    (Here and in the whole paper we adopt the following convention for ${\go sl}_{2}$-triples  $(Y,H,X)$:
    $$ [H,X]=2X,\quad  [H,Y]=-2Y,\quad  [Y,X]=H).$$
    
   Let us denote by ${\go sl}_{2, X}$ the Lie algebra isomophic to ${\go sl}_{2}$ generated by $Y,H,X$. Let us also denote by ${\go o}(Q_{X})$ the orthogonal algebra of $Q_{X}$, viewed as a a subalgebra of ${\go g}\hookrightarrow sp(n,{\bb C})$. It is easy to see that the pair $({\go o}(Q_{X}),{\go sl}_{2, X})$ is a dual pair, in other words these two subalgebras are mutual centralizers in $sp(n,{\bb C})$.
    
    One of the most striking fact attached to the metaplectic representation is the   Howe correspondence. Suppose that $(G_{1},G_{2})$ is a reductive dual pair in $Sp(n,{\bb R})$ (this means just that the groups $G_{1}$ and $G_{2}$ act reductively on ${\bb R}^n$ and are mutual centralizers in $Sp(n,{\bb R})$). Then the pre-images $\widetilde {G_{1}}$ and $\widetilde {G_{2}}$ under the covering map are again mutual centralizers in $\widetilde{Sp}(n,{\bb R})$. For sake of simplicity let us suppose that $G_{1}$ is compact and denote by ${\go g}_{1}$ and ${\go g}_{2}$ the corresponding complexified Lie sub-algebras of $G_{1}$ and $G_{2}$. Then  ${\pi^{\infty}}_{|_{{\go g}_{1}\times {\go g}_{2}}}$ decomposes multiplicity free (see [H-3], Theorem 4.3 p. 190):
   \begin{equation}{\pi^{\infty}}_{|_{{\go g}_{1}\times {\go g}_{2}}}=\oplus (V_{\sigma}\otimes W_{\mu(\sigma)})\end{equation}
    where $V_{\sigma}$ is a finite dimensional irreducible module for ${\go g}_{1}$ and $W_{\mu(\sigma)}$ is an irreducible    module  module for ${\go g}_{2}$. Moreover the correspondence $\sigma \rightarrow \mu(\sigma)$ is a bijection, the so-called Howe correspondence. For the general result by Howe we refer the reader to [H-1].
    
    In the case of the dual pair  $({\go o}(Q_{X}),{\go sl}_{2, X})$, the Howe correspondence is given by
\begin{equation}{\pi^{\infty}}_{|_{ {\go o}(Q_{X})\times {\go sl}_{2, X}}}=\oplus_{k\in {\bb N}}    ({\cal H}_{k}\otimes M(k))\end{equation}
     where ${\cal H}_{k}$ is the space of spherical harmonics of degree $k$ and where $M(k)$ is a  lowest weight module of ${\go sl}_{2, X}$ with lowest weight $k+\frac{n}{2}$ (a discrete series representation) See [H-2], Theorem p. 833.
     
     \begin{rem}\label{rem-engendre-par} It is worth noticing that the  restriction of $\pi^{\infty}$ to ${\go sl}_{2,X}$ completely determines the metaplectic representation $\pi^\infty$. This is due to the fact  that the Lie algebra $sp(n,{\bb C})$ is generated by ${\go sl}_{2,X}$ and by ${\go g}\simeq {\go gl}(n,{\bb C})$ and the fact that the action of ${\go gl}(n,{\bb C})$ is the natural action, twisted by the character $\frac{Tr}{2}$ coming  from the unitarity of   $\pi$ (see $(1.1.1)$).
     \end{rem}
     
     From the preceding remark we see that the infinitesimal metaplectic representation is completely determined by the ${\go sl}_{2}$-triple  $(1.1.3)$. More precisely given a non-degenerate quadratic form $Q$ on ${\bb C}^n$ one  can reconstruct both the symplectic Lie algebra and the metaplectic representation.
     
     It is natural to ask if there exists such a theory if we replace $Q$ by any homogeneous irreducible polynomial $P$. The present paper is devoted to the case where $Q$ is replaced by $\Delta_{0}$, the fundamental relative invariant of an irreducible regular prehomogeneous vector space  of commutative parabolic type $({\go g}, V)$, or equivalently under the Koecher-Tits construction ([Ko],[Ti]),  where $\Delta_{0}$ is the "determinant" function of a simple Jordan algebra over ${\bb C}$. It should also be mentioned that these two categories of objects are equivalent to the hermitian symmetric spaces of tube type.
     
      \vskip 10pt
%%%%%%%%%%%%%%%%%%%%%%%%%%%%%%%%%%%%%%%%%%%
\subsection{}
 
This paper splits into two parts. The present first part, covered by sections $2$ to $7$ is essentially concerned with the study of various algebras of differential operators, in particular the algebra  denoted ${\cal T}_{0}[X,Y]$, which is the generalisation of the associative algebra generated by the ${\go sl}_{2}$-triple $(1.1.3)$. The forthcoming second part will concern  the generalization of the metaplectic repesentation and an infinite dimensional correspondence which is analogue to the correspondence (1.1.5).

  \smallskip
\noindent Let us now give an outline of the Part I. 

 \smallskip

 In Section $2 $ we will briefly recall basic facts concerning  Prehomogeneous Vector Spaces (abbreviated PV), more precisely those which are of   commutative  parabolic type. These objects are in one-to-one correspondence with 3-gradings of simple Lie algebras. In particular we will define the rank of these objects as well as the inductive  construction  of strongly orthogonal roots which leads to the orbit structure. A key ingredient for the sequel is the multiplicity free decomposition of the polynomials on the PV. 
 
 In section 3 we show that the Lie algebra ${\go a} $ generated by $\Delta_{0}(x), \Delta_{0}(\partial)$ and $E$ is infinite dimensional in each of his homogeneous component except if $\Delta_{0}$ has degree 1 or 2. 
 
 In section 4 we show that the associative algebra ${\cal T}$ generated by $\Delta_{0}(x)$, $\Delta_{0}(x)^{-1}$, $\Delta_{0}(\partial)$ and $E$,  which is the biggest algebra of interest to us,  can   be embedded in the Weyl algebra of a complex one-dimensional torus (that is the algebra of differential operators with regular coefficients) tensored by a polynomial algebra in $r-1$ variables, where $r$ is the rank.
 
 Various algebras of invariant differentiable operators related to different group actions occuring in our context are defined and studied in section 5. In this section we also introduce the Harish-Chandra isomorphism for the open orbit which is a symmetric space. This Harish-Chandra isomorphism will be an important tool for us.
 
 In section 6 we describe the center ${\cal Z}({\cal T})$ of ${\cal T}$ and the ideals of ${\cal T}$.We also prove that ${\cal T}$ and some other algebras are noetherian and we compute their Gelfand-Kirillov dimension.
 
 Section 7 contains the main result of this first part. We introduce there the so-called generalized Smith algebras $S({\bf R}, f, n)$ where ${\bf R}$ is a commutative associative algebra over ${\bb C}$, $f\in {\bf R}[t]$ and $n\in {\bb N}$. We show that the algebra ${\cal T}_{0}[X,Y]$ (which is the "polynomial" part of ${\cal T}$) is isomorphic to a quotient of such a Smith algebra.
 
 \vskip 5pt
\noindent {\bf Acknowledgment:} I would like to thank Sylvain Rubenthaler who provided me with a first proof of Proposition \ref{prop.U-contient-pol} which was important for my understanding.
 
 % Dire quelque part dans l'intro que les algebres que nous % considerons ne sont en general pas noetheriennes
     \vskip 10pt  \vskip 10pt
%%%%%%%%%%%%%%%%%%%%%%%%%%%%%%%%%%%%%%%%%%%%%%%%%%%%%%%%
%%%%%%%%%%%%%%%%%%%%%%%%%%%%%%%%%%%%%%%%%%%%%%%%%%%%%%%
\vskip 10pt
\section{Prehomogeneous vector spaces}  
 \vskip10pt
 \vskip10pt

 In this sections we summarize  some results and notations we will need about Prehomogeneous Vector Spaces (abbreviated $PV$).
 %%%%%%%%%%%%%%%%%%%%%%%%%%%%%%%%%%%%%%%%%%%%%%%%%%%%%%
 \subsection{ Prehomogeneous Vector Spaces. Basic definitions and properties} \hfill
 \vskip10pt

 For the general theory of $PV$'s, we refer the reader to the book  of Kimura [Ki]. Let $G$ be an  algebraic group over ${\bb C}$, and let $(G,\rho, V)$ be a   rational representation of $G$ on the (finite dimensional) vector space $V$. Then the triplet $(G,\rho,V)$ is called a {\it Prehomegeneous Vector Space} if the action of $G$ on $V$ has a Zariski open orbit $\Omega\in V$.  The elements in $\Omega$ are called {\it generic}. The $PV$ is said to be {\it irreducible } if  the corresponding representaion is irreducible. The {\it singular set} $S$ of  $(G,\rho,V)$ is defined by $S=V\setminus \Omega$. Elements in $S$ are called {\it singular}. If no confusion can arise we often simply denote the PV by $(G,V)$. We will also   write $g.x$ instead of $\rho(g)x$, for $g\in G$ and $x\in V$. It is easy to see that the condition for a rational representation $(G,\rho,V)$ to be a $PV$ is in fact an infinitesimal condition. More precisely let ${\go g}$ be the Lie algebra of $G$ and let $d\rho$ be the derived representation of $\rho$. Then $(G,\rho,V)$ is a PV if and only if there exists $v\in V$ such that the map:
$$\begin{array}{rcl}
{\go g}&\longrightarrow&V\\
X&\longmapsto&d\rho(X)v
\end{array}$$
is surjective (we will often write $X.v$ instead of $d\rho(X)v$). Therefore we will call $({\go g}, V)$ a $PV$ if the preceding condition is satisfied.

From now on $(G,V)$ will denote a $PV$. A rational function $f$ on $V$ is called a {\it relatively invariant} of    $(G,V)$ if there exists a rational character $\chi $ of $G$ such that $f(g.x)=\chi(g)P(x)$ for $g\in G$ and $x\in V$.  From the existence of an open orbit it is easy to see that a character $\chi$ which is trivial on the isotropy subgroup of an element  $x\in \Omega$ determines a unique    relatively invariant $P$. Let $S_{1},S_{2},\dots,S_{k}$ denote the irreducible components of codimension one of the singular set $S$. Then there exist irreducible polynomials $P_{1}, P_{2},\dots,P_{k}$ such that $S_{i}=\{x\in V\,|\, P_{i}(x)=0\}$. The $P_{i}$'s are unique up to nonzero constants. It can be proved that the $P_{i}$'s are relatively invariants of $(G,V)$ and any relatively invariant $f$ can be written in a unique way $f=P_{1}^{n_{1}}P_{2}^{n_{2}}\dots P_{k}^{n_{k}}$, where $n_{i}\in {\bb Z}$. The polynomials $P_{1}, P_{2},\dots,P_{k}$ are called the {\it fundamental relative invariants} of $(G,V)$. Moreover if the representation $(G,V)$ is irreducible then  there exists at most one irreducible polynomial which is relatively invariant.

 The Prehomogeneous Vector Space  $(G,V)$ is called {\it regular} if there exists a relatively invariant polynomial $P$ whose Hessian $H_{P}(x)$ is nonzero on $\Omega$. If $G$ is reductive, then $(G,V)$ is regular if and only if the singular set $S$ is a hypersurface, or if and only if the isotropy subgroup of a generic point  is reductive. If the $PV$ $(G,V)$ is regular, then the contragredient representation $(G,V^*)$ is again a $PV$.
 
 \vskip 10pt
 %%%%%%%%%%%%%%%%%%%%%%%%%%%%%%%%%%%%%%%%%%%%%%%%%%%%%%
 \subsection{Prehomogeneous Vector Spaces of commutative parabolic type} \hfill
 \vskip 5pt
 The $PV$'s of parabolic type were introduced by the author in [Ru-1] and then developped in his thesis (1982, [Ru-2]). A convenient reference is  the book [Ru-3]. The papers  [M-R-S]  and  [R-S-1]  contain also parts of the results summarized here. Sato and Kimura ([S-K], [Ki]) gave a complete classification of irreducible regular and so called  {\it reduced}  $PV$'s with a reductive group $G$ ({\it reduced} stands for a specific representative in a certain equivalence class, the details are not needed here).  It turns out that most of these $PV$'s are of parabolic type. The class of $PV$'s we are interested in, is a subclass of the full class of parabolic $PV$'s, the so called $PV$'s of commutative parabolic type. Let us now give a brief account of the results which we will need later.
 
Let   $\tilde{\go g}$   be  a simple Lie algebra over ${\bb C}$ satisfying the following two assumptions:

\vskip 5pt
a) There exists a splitting $\tilde{\go g}=V^-\oplus {\go g}\oplus V^+$ which is also a $3$-grading:
 $$\begin{array}{lll}
 [{\go g},V^+]\subset V^+,& [{\go g},V^-]\subset V^-,&[V^-,V^+]\subset {\go g}\\ 
&&\\
{}[V^+,V^+]=\{0\},&[V^-,V^-]=\{0\}.&  
\end{array}$$

b) There exist a semi-simple element $H_{0}\in {\go g}$ and $X_{0}\in V^+, Y_{0}\in V^-$ such that $(Y_{0},H_{0},X_{0})$ is an ${\go s}{\go l}_{2}$-triple.
\vskip 5pt

One can prove, that under the    assumption a) $({\go g}, V^+)$ is an   irreducible PV (here the action of ${\go g}$ on $V^+$ is the Lie bracket). In fact, as we will sketch now, ${\go g}\oplus V^+$ is a maximal parabolic subalgebra of $\tilde{\go g}$ whose nilradical $V^+$ is commutative, this is the reason why these $PV$'s are called of {\it commutative parabolic type}. Assumption b) is equivalent to the regularity of the $PV$ $({\go g}, V^+)$. We will now describe these $PV$'s in terms of roots. Let ${\go j} $ be a Cartan subalgebra of ${\go g}$ which contains $H_{0}$. It is easy to see that ${\go j}$ is also a Cartan subalgebra of $\tilde{\go g}$. Let $\tilde{R}$ be the set of roots of the pair $(\tilde{\go g}, {\go j})$. The set $P$ of roots occuring in ${\go g}\oplus V^+$ is a parabolic subset. Therefore there exists a   set of simple roots $\tilde{\Psi}$, such that  if $\tilde{R}^+$ is the corresponding set of positive roots, then $P \subset \tilde{  R}^+$.  Let $ \omega$ be the highest root in $\tilde{R}$ and  let $R$ be the set of roots of the pair $({\go g}, {\go j})$. Then $\Psi=\tilde{\Psi}\cap R$ is a set of simple roots for $R$ and $\tilde{\Psi}=\Psi \cup \{\alpha_{0}\}$, where $\alpha_{0}$ has coefficient $1$ in $\omega$. What we have done up to now can be performed for all $3$-gradings of $\tilde{\go g}$ (in other words for all splittings of $\tilde{\go g}$ satisfying the assumption a)). It is easy to see that the element $H_{0}$  in assumption b) can be described as the unique element in ${\go j}$ such that 
$$\left\{\begin{array}{cll}
  \alpha(H_{0})&=0& \forall \alpha \in \Psi\\
 \alpha_{0}(H_{0})&=2& 
\end{array}\right.$$
Assumption b) means just that $H_{0}$ is the semi-simple element of an  ${\go s}{\go l}_{2}$-triple. Let $w_{0}$ be the unique element of the Weyl group of $\tilde{R}$ such that $w_{0}(\tilde{\Psi})=-\Psi$. One can show that the preceding condition on $H_{0}$ is equivalent to the condition $w_{0}(\alpha_{0})=-\alpha_{0}$. This leads to an easy classification of the regular $PV$'s of commutative parabolic type. From the preceding discussion we deduce that these objects are in one to one correspondence with connected  Dynkin diagrams where we have circled a root $\alpha_{0}$, which has coefficient $1$ in the highest root and such that $w_{0}(\alpha_{0})=-\alpha_{0}$. In the following table we give the list of these objects and also the corresponding Lie algebra ${\go g}$, and the space $V^+$.

$$\begin{array}{c}
 Table 1\\
 \\
\begin{array}{|c|c|c|c|}
 
\hline
 &\tilde{\go g}& {\go g} & V^+\\
\hline 
 A_{2n+1}&\hbox{\unitlength=0.5pt
\hskip-100pt \begin{picture}(400,30)(0,10)
\put(90,10){\circle*{10}}
\put(85,-10){$\alpha_1$}
\put(95,10){\line (1,0){30}}
\put(130,10){\circle*{10}}
 
\put(140,10){\circle*{1}}
\put(145,10){\circle*{1}}
\put(150,10){\circle*{1}}
\put(155,10){\circle*{1}}
\put(160,10){\circle*{1}}
\put(165,10){\circle*{1}}
\put(170,10){\circle*{1}}
\put(175,10){\circle*{1}}
\put(180,10){\circle*{1}}
 \put(195,10){\circle*{10}}
 
\put(195,10){\line (1,0){30}}
\put(230,10){\circle*{10}}
\put(220,-10){$\alpha_{n+1}$}
\put(230,10){\circle{18}}
\put(235,10){\line (1,0){30}}
\put(270,10){\circle*{10}}
 
\put(280,10){\circle*{1}}
\put(285,10){\circle*{1}}
\put(290,10){\circle*{1}}
\put(295,10){\circle*{1}}
\put(300,10){\circle*{1}}
\put(305,10){\circle*{1}}
\put(310,10){\circle*{1}}
\put(315,10){\circle*{1}}
\put(320,10){\circle*{1}}
\put(330,10){\circle*{10}}
 
\put(335,10){\line (1,0){30}}
\put(370,10){\circle*{10}}
\put(360,-10){$\alpha_{2n+1}$}
%\put(420, 6)  {$A_{2n+1}$}
\end{picture} \hskip -50pt
} 
&{\go s}{\go l}_{n+1}\times{\go s}{\go l}_{n+1}\times {\bb C}&M_{n}({\bb C})\\
&&&\\
\hline
B_{n}&\hbox{\unitlength=0.5pt
\hskip -100pt\begin{picture}(300,30)(-82,0)
\put(10,10){\circle*{10}}
\put(10,10){\circle{18}}
\put(15,10){\line (1,0){30}}
\put(50,10){\circle*{10}}
\put(55,10){\line (1,0){30}}
\put(90,10){\circle*{10}}
\put(95,10){\line (1,0){30}}
\put(130,10){\circle*{10}}
\put(135,10){\circle*{1}}
\put(140,10){\circle*{1}}
\put(145,10){\circle*{1}}
\put(150,10){\circle*{1}}
\put(155,10){\circle*{1}}
\put(160,10){\circle*{1}}
\put(165,10){\circle*{1}}
\put(170,10){\circle*{10}}
\put(174,12){\line (1,0){41}}
\put(174,8){\line (1,0){41}}
\put(190,5.5){$>$}
\put(220,10){\circle*{10}}
\put(240,7){$B_{n}\,(n\geq 2)$}
 \end{picture} 
}
&{\go s}{\go o}_{2n-2}\times {\bb C}&{\bb C}^{2n-2}\\
\hline

 C_{n}&\hbox{\unitlength=0.5pt
\hskip -100pt\begin{picture}(300,30)(-82,0)
\put(10,10){\circle*{10}}
\put(15,10){\line (1,0){30}}
\put(50,10){\circle*{10}}
\put(55,10){\line (1,0){30}}
\put(90,10){\circle*{10}}
\put(95,10){\circle*{1}}
\put(100,10){\circle*{1}}
\put(105,10){\circle*{1}}
\put(110,10){\circle*{1}}
\put(115,10){\circle*{1}}
\put(120,10){\circle*{1}}
\put(125,10){\circle*{1}}
\put(130,10){\circle*{1}}
\put(135,10){\circle*{1}}
\put(140,10){\circle*{10}}
\put(145,10){\line (1,0){30}}
\put(180,10){\circle*{10}}
\put(184,12){\line (1,0){41}}
\put(184,8){\line(1,0){41}}
\put(200,5.5){$<$}
\put(230,10){\circle*{10}}
\put(230,10){\circle{18}}
\put(250, 6){$C_{n}$}
 
\end{picture}
}
&{\go g}{\go l}_{n}&{MS}_{n}({\bb C})\\
 \hline
D_{n}^1&\hbox{\unitlength=0.5pt
\hskip -90pt\begin{picture}(340,40)(-82,8)
\put(10,10){\circle*{10}}
\put(10,10){\circle{18}}
\put(15,10){\line (1,0){30}}
\put(50,10){\circle*{10}}
\put(55,10){\line (1,0){30}}
\put(90,10){\circle*{10}}
\put(95,10){\line (1,0){30}}
\put(130,10){\circle*{10}}
\put(140,10){\circle*{1}}
\put(145,10){\circle*{1}}
\put(150,10){\circle*{1}}
\put(155,10){\circle*{1}}
\put(160,10){\circle*{1}}
\put(170,10){\circle*{10}}
\put(175,10){\line (1,0){30}}
\put(210,10){\circle*{10}}
\put(215,14){\line (1,1){20}}
\put(240,36){\circle*{10}}
\put(215,6){\line (1,-1){20}}
\put(240,-16){\circle*{10}}
\put(260,6){$D_{n}\,(n\geq 4)$}
 \end{picture}
}
&{\go s}{\go o}_{2n-1}\times {\bb C}&{\bb C}^{2n-1}\\

&&&\\
\hline
 D_{2n}^2  &\hbox{\unitlength=0.5pt
\hskip-65pt\begin{picture}(400,35)(-82,10)
\put(10,0){\circle*{10}}
\put(15,0){\line (1,0){30}}
\put(50,0){\circle*{10}}
\put(55,0){\line (1,0){30}}
\put(90,0){\circle*{10}}
\put(95,0){\line (1,0){30}}
\put(130,0){\circle*{10}}
\put(140,0){\circle*{1}}
\put(145,0){\circle*{1}}
\put(150,0){\circle*{1}}
\put(155,0){\circle*{1}}
\put(160,0){\circle*{1}}
\put(170,0){\circle*{10}}
\put(175,0){\line (1,0){30}}
\put(210,0){\circle*{10}}
\put(215,4){\line (1,1){20}}
\put(240,26){\circle*{10}}
\put(240,26){\circle{18}}
\put(215,-4){\line (1,-1){20}}
\put(240,-26){\circle*{10}}
\put(250,0){$D_{2n}$ ($ n\geq 2$)}
\end{picture}
}
&{\go g}{\go l}_{2n}({\bb C})&AS_{2n}({\bb C})\\
&&&\\
&&&\\
\hline
 E_7&\hbox{\unitlength=0.5pt
 \begin{picture}(350,35)(-82,-10)
\put(10,0){\circle*{10}}
\put(15,0){\line  (1,0){30}}
\put(50,0){\circle*{10}}
\put(55,0){\line  (1,0){30}}
\put(90,0){\circle*{10}}
\put(90,-5){\line  (0,-1){30}}
\put(90,-40){\circle*{10}}
\put(95,0){\line  (1,0){30}}
\put(130,0){\circle*{10}}
\put(135,0){\line  (1,0){30}}
\put(170,0){\circle*{10}}
\put(175,0){\line (1,0){30}}
\put(210,0){\circle*{10}}
\put(210,0){\circle{18}}
\put(240,-4){$E_7$}
\end{picture}
}
 &E_{6}\times {\bb C}&{\bb C}^{27}\\
&&&\\
&&&\\
\hline
\end{array}
\end{array}$$
\vskip 10pt

\noindent We need to get more inside the structure of the regular $PV$'s of commutative parabolic type. First let us define the rank of such a $PV$. Let $\tilde{R}_{1}$ be the set of roots which are orhogonal to $\alpha_{0}$ (this is also the set of roots which are strongly orthogonal to $\alpha_{0}$). The set  $\tilde{R}_{1}$ is again a root system as well as $R_{1}=\tilde{R}_{1}\cap R$. Define 
$${\go j}_{1}=\sum_{\alpha\in \tilde{R}_{1}}{\bb C}H_{\alpha},\qquad  \tilde{\go g}_{1}={\go j}_{1}\oplus \sum_{\alpha\in \tilde{R}_{1}}\tilde{\go g}^{\alpha}.$$
Then $\tilde{\go g}_{1}$ is a semi-simple Lie algebra, ${\go j}_{1} $ is a Cartan subalgebra of $\tilde{\go g}_{1}$ and   the set  roots of  $(\tilde{\go g}_{1}, {\go j}_{1})$ is  $\tilde{R}_{1}$.  Moreover if we set ${\go g}_{1}= \tilde{\go g}_{1}\cap {\go g}$, $V^+_{1}=\tilde{\go g}_{1}\cap V^+  $ and $V^-_{1}=\tilde{\go g}_{1}\cap V^- $, then ${\go j}_{1}$ is also a Cartan subalgebra of ${\go g}_{1}$ and we have 
$$\tilde{\go g}_{1}=V^-_{1}\oplus {\go g}_{1}\oplus V^+_{1}.$$
The key remark is that if we start with $(\tilde{\go g}, {\go g}, V^+)$ which satisfies assumption a), the preceding splitting of ${\go g}_{1}$ is again a $3$-grading satisfying assumptions a), in other words $({\go g}_{1}, V^+_{1})$ is again a $PV$ of commutative parabolic type (except that the algebra  $\tilde{\go g}_{1}$ may be  semi-simple, not necessarily simple). Moreover if $(\tilde{\go g}, {\go g}, V^+)$ satisfies also b), then the same is true for $(\tilde{\go g}_{1}, {\go g}_{1}, V^+_{1})$. 
\vskip 5pt

\noindent  Let $\alpha_{1}$ be the root which plays the role of $\alpha_{0}$ for the new $PV$ of commutative parabolic type $(\tilde{\go g}_{1}, {\go g}_{1}, V^+_{1})$. We can apply the same procedure, called the {\it descent}, to $(\tilde{\go g}_{1}, {\go g}_{1}, V^+_{1})$ and so on and then we will obtain inductively a sequence
$$\dots \subset (\tilde{\go g}_{k}, {\go g}_{k}, V^+_{k})\subset \dots \subset(\tilde{\go g}_{1}, {\go g}_{1}, V^+_{1})\subset (\tilde{\go g}  , {\go g} , V^+ )$$
of $PV$'s of commutative parabolic type. This sequence stops because, for dimension reasons, there exists an integer $n$ such that $\tilde{R}_{n} \neq \emptyset$ and such that $\tilde{R}_{n+1}=\emptyset$. 

The integer $n+1$ is then called the {\it rank} of   $({\go g},V^+)$.

Let us denote by $\alpha_{0},\alpha_{1},\dots,\alpha_{n}$ the set of strongly orthogonal roots occuring in $V^+$ which appear in the descent (the rank is also the number of elements in this sequence, it can be characterized as the maximal number of strongly orthogonal roots occuring in $V^+$).  One proves also that if the first $PV$ $(\go{g},  V^+)$ is   regular ({\it  i.e.} if it satisfies b)), then $V^+_{n}={\bb C}X_{\alpha_{n}}$, where $X_{\alpha_{n}}\in \tilde{\go g}, \,\, X_{\alpha_{n}}\neq 0$.  

Let $\tilde{G}$ be the adjoint group of the $\tilde{\go g}$ and let $G$ be the analytic subgroup of $\tilde{G}$ corresponding to ${\go g}$. The group $G$ is also  the centralizer of $H_{0}$ in $\tilde{G}$. The representation $(G, \Ad, V^+)$ is then an irreducible        $PV$. Let $G_{n}$ be the subgroup of $G$ corresponding to ${\go g}_{n}$. The descent process described before leads to a sequence of $PV$'s:
$$(G_{n},V^+_{n})\subset (G_{n-1},V^+_{n-1}\subset\dots\subset (G_{1},V^+_{1})\subset (G,V^+) \eqno(2-2-1)$$
The orbital structure of $(G,V^+)$ can the be described as follows. Let us denote as usual by $X_{\gamma}$ a non zero element of $\tilde{\go g}^{\gamma}$.

Define
\begin{align*}I_{0}^+&=X_{\alpha_{0}}, \, I_{1}^+=X_{\alpha_{0}}+X_{\alpha_{1}}, \dots,\,\,I_{k}^+=X_{\alpha_{0}}+X_{\alpha_{1}}+\dots+X_{\alpha_{k}},\dots,\\
I_{n}^+=I^+&=X_{\alpha_{0}}+X_{\alpha_{1}}+\dots+X_{\alpha_{n}}.
\end{align*}
Then the set  
$$\{0, I_{0}^+,I_{1}^+,\dots, I_{n}^+=I^+\}$$
is a set of representatives of the $G$-orbits in $V^+$ (there are $rank(G,V^+)+1$ orbits). The orbit   $G.I^+$ is the open orbit $\Omega^+\subset V^+$.

The Killing form $\tilde{B}$ of $\tilde{\go g}$ allows us to identify $V^-$ to the dual space of $V^+$ and the  representation $(G,V^-)$ becomes  then the dual $PV$ of $(G,V^+)$.
One can similarly perform a descent on the $V^-$ side, and obtain a sequence 
$$(G_{n},V^-_{n})\subset (G_{n-1},V^-_{n-1}\subset\dots\subset (G_{1},V^-_{1})\subset (G,V^-)  \eqno(2-2-2)
$$
of $PV$'s, where the groups are the same as in $(2-2-1)$. The $PV$  $(G_{i},V^-_{i})$ is dual to $(G_{i},V^+_{i})$. The set of elements 
$$\{0, I_{0}^-,I_{1}^-,\dots, I_{n}^-=I^-\}$$
where $I^-_{i}=X_{-\alpha_{0}}+X_{-\alpha_{1}}+\dots+X_{-\alpha_{i}}$ is a set of representatives of the $G$-orbits in $V^-$. The orbit   $G.I^-$ is the open orbit $\Omega^-\subset V^-$. We will always choose the elements $X_{-\alpha_{i}}$ such that $(X_{-\alpha_{i}},H_{\alpha_{i}},X_{\alpha_{i}})$ is a ${\go s}{\go l}_{2}$-triple. If the $PV$ $({\go g}, V^+)$ satisfies the assumptions a) and b), then $H_{0}=H_{\alpha_{0}}+H_{\alpha_{1}}+\dots+H_{\alpha_{n}}$ and $(I^-,H_{0},I^+)$ is a ${\go s}{\go l}_{2}$-triple. More generally, under the same hypothesis, if $H_{i}=H_{\alpha_{0}}+H_{\alpha_{1}}+\dots+H_{\alpha_{i}}$, then $(I^-_{i},H_{i},I^+_{i})$ is a ${\go s}{\go l}_{2}$-triple.
\vskip 5pt

We suppose from now on that $(G,V^+)$ is regular. Remember that this means that it satisfies assumption b).
Let   $\Delta_{0}$ be the unique irreducible polynomial polynomial on $V^+$ which is relatively invariant under the action of $G$. Let $\Delta_{1}$ be the unique irreducible polynomial on $V^+_{1}$ which is relatively invariant under the group $G_{1}$. We have   $V^+=W^+_{1}\oplus V^+_{1}$ where $W^+_{1}$ is the sum of the root spaces of $V^+$ which do not occur in $V^+_{1}$. Therefore for $x=y_{1}+x_{1}$ ($x_{1}\in V^+_{1}, y_{1}\in W^+_{1}$), we can define $\Delta_{1}(x)=\Delta_{1}(x_{1})$ and hence the polynomial $\Delta_{1}(x)$ can be viewed as a polynomial on $V^+ $. Inductively we can define a sequence $\Delta_{0},\Delta_{1},\dots,\Delta_{n}$ of irreducible polynomials on $V^+$, where the polynomial $\Delta_{i}$ depends only on the variables in $V^+_{i}$ and is, in general,  not relatively invariant under   $G$  , but under   $G_{i}\subset G$. It can be shown that $\partial^\circ (\Delta_{i})=n+1-i=\rank(G,V^+)-i$.

Let $H\subset G$ be the isotropy subgroup of $I^+$ and let ${\go h}\subset {\go g}$ be its Lie algebra. Another striking fact concerning the irreducible regular $PV$'s of commutative parabolic type is that the open orbit $G/H\simeq \Omega^+$ is a symmetric space. This means that ${\go h}$ is the fixed points set of an involution of ${\go g}$ (this involution can be shown to be $\exp\ad(I^+)\exp\ad (I^-)\exp\ad (I^+)$).  Let us denote by $B^-$ the Borel subgroup of $G$ defined by $- {\Psi}$. One can show that  $V^+$ is already a $PV$ under the action of $B^-$. More precisely the fundamental relative invariants of $(B^-, V^+)$ are the polynomials $\Delta_{0},\Delta_{1},\dots,\Delta_{n}$. The open $B^-$-orbit in $V^+$ is the set 
$${\cal O}^+=\{x\in V^+\,|\, \Delta_{0}(x)\Delta_{1}(x)\dots\Delta_{n}(x)\neq0\}.$$
Symmetrically, if $B^+$ is the Borel subgroup of $G$ defined by ${\Psi}$, the representation $(B^+,V^-)$ is also a $PV$, and the fundamental  relatively invariant polynomials is  a set  
$$\{\Delta^*_{0},\Delta^*_{1},\dots,\Delta^*_{n}\} $$
of irreducible polynomials where $\partial^\circ (\Delta^*_{i})=n+1-i=\rank(G,V^+)-i$. Of course these polynomial are obtained by a descent process similar to the one described before on $V^+$. Moreover the polynomial $\Delta^*_{0}$ is the fundamental relatively invariant of $(G,V^-)$.

Let $\tilde{B}$ be the Killing form on $\tilde{\go g}$. Then for any polynomial $P^*$ on $V^-$, we define a differential operator $P^*(\partial)$ with constant coefficients on $V^+$ by the formula
$$P^*(\partial)e^{\tilde{B}(x,y)}=P(y)e^{\tilde{B}(x,y)}\text{  for  }x\in V^+, y\in V^-.\eqno(2-2-3)$$
The differential operators $\Delta_{0}^*(\partial)$ and $\Delta_{0}(x)$ (multiplication by the function $\Delta_{0}(x)$) will play an important role in this paper. From the fact that $(B^-,V^+)$ has an open orbit it is easy to see that the space ${\bb C}[V^+]$  of  polynomialson $V^+$, viewed as a representation space of $G$ decomposes multiplicity free.  For ${\bf a}=(a_{0},a_{1},\dots,a_{n})\in {\bb N}^{n+1}$, let $V_{\bf a}$ be the $G$-submodule of ${\bb C}[V^+]$ generated by $\Delta_{\bf a}(x)=\Delta_{0}^{a_{0}}(x) \Delta_{1}^{a_{1}}(x)\dots\Delta_{n}^{a_{n}}(x)$. Then $V_{\bf a}$ is an irreducible $G$-module with highest weight vector $\Delta_{\bf a}(x)$ (with respect to $-\Psi$). Moreover the representations $V_{\bf a}$ and $V_{\bf a'}$ are not equivalent if ${\bf a}\neq{\bf a'}$ and 
  $${\bb C}[V^+]=\oplus _{{\bf a}\in {\bb N}^{n+1}}V_{\bf a }\eqno(2-2-4)$$
  \begin{rem}\label{rem.PVcom=alg.Jordan=tube}
  Let $G/K$ be a hermitian symmetric space, let ${\go g}_{\bb C}={\go p}^-_{\bb C}\oplus {\go k}_{\bb C}\oplus {\go p}^+_{\bb C}$ be the usual decomposition of the complexified Lie algebra of $G$. Then of course the preceding slitting is a 3-grading of ${\go g}_{\bb C}$. Moreover this 3-grading verifies also assumption b) at the beginning of section 2.2. if and only if $G/K$ is of tube type. Conversely it can be shown that any    $PV$ of commutative parabolic type  can be obtained this way from a hermitian symmetric space of tube type.
  
  Similarly if $J$ is a simple Jordan algebra over ${\bb C}$, and if $struc(J)$ is its structure algebra, then from the Koecher-Tits   construction  ([Ko], [Ti]) it is know that one can put a Lie algebra structure on $J\oplus struc(J)\oplus J$, and this splitting is a 3-grading which verifies assumption b). Conversely  if  $(G,V^+)$ is a $PV$ of commutative parabolic type, then one can define on $V^+$ a Jordan product which makes $V^+$ into simple Jordan algebra (see also [F-K]).
 \end{rem}
  
  \vskip 10pt
 %%%%%%%%%%%%%%%%%%%%%%%%%%%%%%%%%%%%%%%%%%%%%%%%%%%%%%%%%%%%%%%%%%%%%%%%%%%%%%%%%%%%%%
\section{  Some algebras of differential operators }
\vskip 10pt
\vskip 10pt
%%%%%%%%%%%%%%%%%%%%%%%%%%%%%%%%%%%%%%%%%%% 
\subsection{Definitions and basic properties.}\hfill
\vskip 10pt

  For any smooth affine algebraic variety $  M$, we shall
denote by $  {\bb C}   [M]$ the algebra of regular functions on
$  M$. If $  M =U$ is a vector space, then ${\bb C} [U]$ is the
algebra of polynomials on $U$. As $S^+=V^+\setminus \Omega^+$ is the  hypersurface defined by $\Delta_{0}$, the algebra ${\bb C} [\Omega ^+]$ of
regular functions on
$\Omega ^+$ is the algebra of  fractions of the form ${P\over
\Delta_0^k}$, $P\in {\bb C} [V^+]$.

\vskip 10pt
 We shall denote by ${\bf D}(M)$ the algebra of
differential operators on
$M$. For example if
$U$ is a vector space, the algebra ${\bf D}(U)$ is  the Weyl algebra
of $U$. If $D_{1}, D_{2}, \dots, D_{k}$ is a finite family of elements in ${\bf D}(M)$, we will denote by ${\bb C}[D_{1}, D_{2}, \dots,D_{k}]$ the subalgebra (with unit)   of ${\bf D}(M)$ generated by $D_{1}, D_{2}, \dots, D_{k}$. Of course ${\bb C}[D_{1}, D_{2}, \dots,D_{k}]$ is also the set of linear combinations of monomials in the noncommutative generators $D_{1}, D_{2}, \dots, D_{k}$.

\vskip 10pt
If $g$ is any algebraic diffeomorphism of the variety $M$, and   $f\in    {\bb C}   [M] $, we define the function $\tau_{g}f \in    {\bb C}   [M]$ by  $\tau_{g}f(x)=f(g^{-1}x),\text{ } x\in M$. Then for any $D\in {\bf D}(M) $, we also define $\tau_{g}D \in  {\bf D}(M) $ by 
$\tau_{g}D =\tau_{g}\circ D\circ \tau_{g^{-1}} $.

\vskip 10pt

   \noindent   We define two differential operators in
${\bf D}(V^+)$ by

\begin{align*} &X=\Delta_0(x) \qquad \text{ (multiplication by }  
\Delta_0(x))\\ 
  &Y=\Delta_0^*({\partial }).  \qquad \text{ (defined by $(2-2-3)$ }  
  \end{align*}

   \noindent      We will also consider the Euler operator $E$ on $V^+$
which is defined on any $P\in {\bb C} [V^+]$, by 

$$  
 EP(x)= {\partial\over{\partial t}} P(tx)_{t=1} =P'(x)x.$$
 
 \noindent    Let us recall from [R-S-1] page 424, rel(4-4),   that if $\chi_{0}$ is the character of $\Delta_{0}$, we have:
  
  \begin{align*} \forall g \in G, \hskip 10pt  & \tau_{g}X=\chi _{_{0}}(g^{-1})X\\ 
 & \tau_{g}Y=\chi _{_{0}}(g )Y   \quad \quad   \quad \quad  (3-1-1)\\
 &\tau_{g}E=E. \end{align*} 

\vskip 10pt
       \noindent       Let ${\go a}$ be the Lie sub-algebra of ${\bf D}(V^+)$
generated by $X$, $Y$ and $E$ and let
${\cal A}$ be the associative sub-algebra of ${\bf D}(V^+)$ generated by
$X$, $Y$, $E$ and ${\bb C} $ (${\cal A}={\bb C} [X, Y,E]$).

\vskip 10pt
        \noindent       The algebra  ${\bf D}(\Omega ^+)$ is the algebra of
differential operators on $V ^+$ whose coefficients lie in ${\bb
C} [ \Omega^+]$ . Of course one has ${\bf D}( V ^+)\subset {\bf
D}(\Omega ^+)$.

\vskip 5pt
     \noindent         We also introduce  another differential operator $X^{-1}\in
{\bf D}(\Omega ^+)$:

$$X^{-1}=\Delta_0(x)^{-1} \qquad {\rm (multiplication \,\,by}\,\, 
\Delta_0(x)^{-1}).$$

    \noindent          Let ${\cal T}$ be the associative sub-algebra of ${\bf
D}(\Omega ^+)$ generated by $X$, $Y$,
$E$, $X^{-1}$ and ${\bb C} $ (${\cal T}={\bb C}
[E, X,X^{-1},Y])$. The inclusions 

$${\go a}\subset {\cal A} \subset {\cal T}$$ 
are obvious.
 
\vskip 5pt
        \noindent       Remember that the space ${\bb C} [V^+]$   decomposes multiplicity
free (see $ (2-2-4)$) under the action of $G$:

$${\bb C} [V^+]=\bigoplus_{{\bf a} \in {\bb N} ^{n+1}} V_{\bf
a}\eqno (3-1-2) $$
where $V_{\bf a}=V_{a_0,\ldots,a_n} $ is the irreducible $G$-submodule
generated by
$\Delta_0^{a_0}\ldots\Delta_n^{a_n}=\Delta^{\bf a}$.  Sometimes it will
be convenient to use the following notations. If ${\bf
a}=(a_0,\ldots,a_n)\in {\bb N} ^{n+1}$, then we will write ${\bf
a}=(a_0,\tilde{\bf a})$, where $\tilde{\bf a}=(a_1\ldots,a_n)\in
{\bb N} ^n$, and  we will denote $V_{(0,\tilde{\bf a})}$ by
$V_{\tilde{\bf a}}$.

\vskip 5pt
   \noindent           It is easily seen that $V_{\bf
a}=\Delta_0^{a_0}V_{0,a_1,\ldots,a_n}=\Delta_0^{a_0}V_{\tilde{\bf a} }$.

\vskip 5pt 
     \noindent         More generally the space ${\bb C} [\Omega^+]$   also
decomposes multiplicity free under $G$:

$${\bb C} [\Omega^+]=\bigoplus_{{\bf a} \in {\bb Z}\times{\bb
N} ^{n }} V_{\bf a}\eqno (3-1-3)$$
with the same definition for the spaces $V_{\bf a}$.
\vskip 5pt

     \noindent         We make the convention that if $ {\bf z}=
(z_0,\ldots,z_n)$ is any set of variables, then for any integer $p\in {\bb Z}$, ${\bf z}+p=
(z_0+p,z_1,\ldots,z_n)$ and $\tilde{\bf z}=(z_{1},\dots,z_{n})$

\vskip 5pt
       \noindent         It is easy to see that the operator $X$ maps $V_{\bf a}$
onto $V_{{\bf a}+1}$  and is a  $G'$  and ${\go g}'$  equivariant
isomorphism, where $G'$ is the derived group of $G$ and ${\go g}'$ its
Lie algebra. The same way  the operator  $Y$ maps $V_{\bf a}$ into
$V_{{\bf a}-1}$ and is
$G'$ and
${\go g}'$ equivariant ([R-S-1], Lemme 4.3 page 424). 
\vskip 5pt 

     \noindent         If $P\in V_{\bf a}$, then $YP=b_Y({\bf
a}){P\over{\Delta_0}}= b_Y({\bf a})X^{-1}P$ ([R-S-1], Remarque 4.5.), where
$b_Y$ is a polynomial in $n+1$ variables, which we will call the
{\it Bernstein-Sato polynomial} or the {\it b-function} of the operator $Y$. If ${\bf
X}=(X_0,\ldots,X_n) $ then we have, under a suitable normalization \footnote{The polynomials $\Delta_{0}$ and $\Delta_{0}^*$ are only defined up to nonzero constants. We suppose here that $\Delta_{0}(I^+)=1$ and then we choose $\Delta_{0}^*$ such that $(3-1-4)$ holds }

$$b_Y({\bf X})=\prod _{j=0}^n(X_0+\cdots+X_j+j{d\over2}), \eqno (3-1-4)$$
where ${d\over2}={{k-(n+1)}\over {n(n+1)}}$, $k=\dim V^+$ (it can be noted that this constant is the same as the constant also denoted by $d$ in the theory of simple Jordan algebras over ${\bb C}$, {\it cf.} Remark \ref{rem.PVcom=alg.Jordan=tube}). 
\vskip 5pt

         \noindent      This explicit computation of the polynomial $b_Y$ has been
obtained by several authors, using distinct methods (see [B-R], [F-K],
[K-S], [Wa]). As a consequence of this, it is easy to see that the operator
$Y$ has a kernel ${\cal H}[V^+]$ (the "harmonic" polynomials)  in
${\bb C} [V^+]$ which can be described the following way:
\vskip 5pt

$${\cal H}[V^+]=\bigoplus_{{\bf a} \in {\{0\}}\times{\bb N} ^{n }}
V_{\bf a}=\bigoplus_{\tilde{\bf a}
\in  {\bb N} ^{n }}V_{\tilde{\bf a}.}
  \eqno (3-1-5 )$$

\vskip 5pt 

  \noindent    The preceeding decomposition has been obtained, without the
explicit knowledge of $b_Y,
$ by Rubenthaler-Schiffmann ([R-S-1] Th\'eor\`eme 4.4.) and Upmeier ([Up],
Theorem 2.6.). This decomposition means that the restriction of
$Y$ to
$V_{\bf a}$ is an isomorphism onto $V_{{\bf a}-1}$ except if ${\bf a}\in
{\{0\}}\times{\bb N} ^{n }$.

 \vskip 5pt

   \noindent   We will now define  ${\bb Z}$-gradings  on ${\go a}$,
${\cal A}$ and ${\cal T}$. Recall that the      operators $E$, $X$,
$X^{-1}$ and $Y$ act naturally on the space 
${\bb C} [\Omega ^+]$ . Let
${\bb C} [V^+]=\oplus_{n=0}^{\infty}{\bb C}^n [V^+]$ be the
natural grading  of the polynomials by the homogeneous degree. The
homogeneous degree also defines a grading on the regular functions on
$\Omega ^+$: ${\bb C} [\Omega ^+]= \oplus _{n\in {\bb Z}}{\bb
C} ^n[\Omega ^+]$.
\vskip 5pt

\noindent   Define   for each
$p\in {\bb Z}$:

\begin{equation}
\begin{array}{rl}
 {\go a}_p &=\{D\in {\go a}\,|\,D:V_{\bf a}\mapsto V_{{\bf
a}+p},\,\,\forall {\bf a}\in {\bb N} ^{n+1}\}\\ 
 &= \{D\in {\go a}\,|\,D: {\bb C}^m [ \Omega ^+]\mapsto  {\bb
C}^{m+(n+1)p} [ \Omega ^+], \,\,\forall { m}\in  {\bb N} \}\\
&=\{D\in {\go a}\,|\,[E,D]=p(n+1)D\}\\ &{}\\
 {\cal A }_p&=\{D\in {\cal  A}\,|\,D: V_{\bf a}\mapsto V_{{\bf a}+p},
\,\,\forall {\bf a}\in {\bb N} ^{n+1}\} \\
&=\{D\in {\cal  A}\,|\,D:
{\bb C}^m [ \Omega ^+] \mapsto  {\bb C}^{m+(n+1)p} [\Omega ^+],
\,\,\forall { m}\in  {\bb N}  \}\\
&=\{D\in {\cal
A}\,|\,[E,D]=p(n+1)D\}\\
&{}\\ {\cal T }_p&=\{D\in {\cal  T}\,|\,D: V_{\bf
a}\mapsto V_{{\bf a}+p}, \,\,\forall {\bf a}\in {\bb N} ^{n+1}\} \\
&=\{D\in {\cal  T}\,|\,D: {\bb C}^m [ \Omega ^+] \mapsto  {\bb
C}^{m+(n+1)p} [\Omega ^+], \,\,\forall { m}\in  {\bb N}  \}\\
&=\{D\in {\cal T}\,|\,[E,D]=p(n+1)D\}
\end{array}\tag{3-1-6 }
\end{equation}

\vskip 5pt
     \noindent  Of course one has $E\in {\go a}_0$,  $X\in {\go a}_1$,
$X^{-1}\in  {\cal   T}_{-1}$ and $Y\in {\go
a}_{-1}$. Moreover

\begin{equation}
\begin{array}{rl}{\go a} = \oplus_{p\in {\bb N}}{\go a}_p, \text{  \,\,  } {\cal A} = \oplus_{p\in {\bb N}}{\cal A}_p,  \text {Ê\,\,\,ÊÊ}{\cal T} =
\oplus_{p\in {\bb N}}{\cal T}_p\\\end{array}\tag{3-1-7 }
\end{equation}

\vskip 5pt

\noindent  and these decompositions are ${\bb Z}$-gradings of Lie algebras or associative
algebras.
 
\vskip 5pt

 \noindent   An element $D$ belonging to  ${\cal T}_p$ defines a
${\go g}'$-equivariant map from $V_{\bf a}$ to $V_{{\bf a}+p}$. As
$V_{\bf a}$ and $V_{{\bf a}+p}$ are irreducible ${\go g}'$-modules, there
exists a constant $b_D({\bf a})$ such that for each  $P\in V_{\bf a}$, we
have $DP=b_D({\bf a})\Delta_0^p P$. It is easy to see that $b_D$ is a
polynomial (in
$n+1$ variables).

\vskip 5pt 
\begin{definition}\label{def-polynome-BS} If $D$ is a differential operator in ${\cal T}_{p}$ the polynomial $b_{D}(s)$ defined before is called  the   Bernstein-Sato
polynomial  or the   b-function  of the operator $D$.
\end{definition}

 \vskip 5pt\vskip 5pt
%%%%%%%%%%%%%%%%%%%%%%%%%%%%%%%%%%%
\subsection{ First results concerning ${\go a}$, $ {\cal   A}$,
and
$ {\cal   T}$.}\hfill
 
 \vskip 5pt
%%%%%%%%
\begin{theorem}\label{th-algebre-infinie} If the degree of the polynomial $\Delta_0$ is not
equal to $1$ or $2$ (that is if
$n
 > 1$), then the Lie algebra ${\go a}$ is infinite dimensional. More
precisely for any $p\in {\bb N}
$ the space ${\go a}_p$ is infinite dimensional. As a consequence ${\cal
A}_p$ and ${\cal T}_p$ are also infinite dimensional. If
$\Delta_0$ is a quadratic form, then the Lie algebra ${\go a}$ is
isomorphic to a semi-direct product  ${\go
sl}_2({\bb C} ){\bb o} {\go c}$, where ${\go c}$ is one dimensional 
central. If
$\Delta _0  $ is  of degree one then the Lie algebra ${\go a}$ is
isomorphic to the semi-direct product $H_3 {\bb o} {\bb C} E$ where
$H_3$ is the three dimensional Heisenberg Lie algebra generated by $X$
and $Y$. 
\end{theorem}

\begin{proof}\hfill

   \noindent  If the degree of $\Delta_0$ = $1$, then $\dim V^+=1$ and it
is clear that the two operators $x$ and
${\partial\over{\partial x}}$ generate the three dimensional Heisenberg
Lie algebra. If $\Delta_0$ is a quadratic form it is also well known that
$[Y,X]=E+{k\over2}$ ($k=\dim V^+)$ and therefore the Lie
algebra generated by $X$ and $Y$ is isomorphic to ${\go
sl}_2({\bb C} ).$ (see for example [Ra-S] or [H-3], p. 199). This fact is in
some sense the key point of the construction of the Weil
representation. It implies the result for this case.

\vskip 5pt
 \noindent    Suppose now that the degree of $\Delta_0$ is $\geq 3$.

   \noindent    For any $D\in {\cal A}$ we shall denote by
$\partial^{\circ}b_D$ the degree of $b_D$ in the $a_0$ variable. The
formula  $(3-1-4)$  shows that $\partial^{\circ}b_Y=n+1$.  Let
$D\in {\cal A}_0$ and suppose that  
$\partial^{\circ}b_D=p$   , define
$\widetilde{D}=[Y,D]\in {\cal A}_{-1}$. Let $P\in V_{\bf a}$. We have

\begin{equation}
\begin{array}{rl} \widetilde{D}P&=YDP-DYP=(b_Y({\bf a})b_D({\bf a})- b_D({{\bf
a}-1})b_Y({\bf a})){P\over{\Delta_0}}\cr &=b_Y({\bf a})(b_D({{\bf
a}})-b_D({{\bf a}-1})){P\over {\Delta_0}} \end{array} \end{equation}

\vskip5pt \noindent    This proves that $b_{\widetilde{D}}=b_Y({\bf a})(b_D({{\bf
a}})-b_D({{\bf a}-1})) $ and therefore
$\partial^{\circ}b_{\widetilde{D}}=p+n$.

   \noindent    Let
$\widetilde{\widetilde{D}}=[X,[Y,D]]=[X,\widetilde{D}]$. Let us  compute $\widetilde{\widetilde{D}}P$ for $P\in V_{\bf a}$.
We get:
$$\widetilde{\widetilde{D}}P=X\widetilde{D}P-\widetilde{D}XP=(b_{\widetilde{D}}({\bf
a})-b_{\widetilde{D}}({\bf a}+1))P.$$

  \noindent     Therefore $b_{\widetilde{\widetilde{D}}}({\bf
a})=b_{\widetilde{D}}({\bf a})-b_{\widetilde{D}}({\bf a}+1)$ and hence
$\partial^{\circ}b_{\widetilde{ \widetilde{D}}}=p+n-1$.

   \noindent    Remark that if $n=1$ (i.e. if $\Delta_0$ is a quadratic
form), then
$\partial^{\circ}b_{\widetilde{ \widetilde{D}}}=  \partial^{\circ}b_{ D
}  $ and if $n=0$ (i.e. if $\Delta_0$ is a linear 
form), then 
$\partial^{\circ}b_{\widetilde{ \widetilde{D}}}=  \partial^{\circ}b_{ D
}-1  $.

\vskip 5pt
    \noindent   Define now inductively the operators $H_q\in {\go a}_0$ by
$H_1=[X,Y]$,
$H_q=[X,[Y,H_{q-1}]]$. As
$b_{H_1}({\bf a})=b_Y({\bf a})-b_Y({\bf a}+1)$ has degree $n$ in $a_0$,
one has
$\partial^{\circ}b_{H_q}=(q-1)(n-1)+n$.

\vskip 5pt
   \noindent    Therefore there exists operators in ${\go a}_0$ whose
Bernstein-Sato polynomial  has an arbitrarily large degree in
$a_0$. Hence $\dim {\go a}_0=+\infty$.

\vskip 5pt
     Let $p>0$ and let $D\in {\go a}_p$. Define $D_1=[X,D]$. An
easy computation shows that
$b_{D_1}({\bf a})=b_D({\bf a})-b_D({\bf a}+1)$. As we have just proved
that there exists operators
$D\in {\go a}_0$ such $\partial^{\circ}b_ {D}   $ is arbitrarily large,
the preceeding calculation shows inductively that in ${\go a}_p$ $(p>0)$
too, there exists operators whose Bernstein-Sato  polynomial has an
arbitrarily large degree in $a_0$. Therefore $\dim {\go a}_p=+\infty$ for
$p>0$.

\vskip 5pt
     If $p\leq 0$ and if $D\in {\go a}_p$, define $D_1=[Y,D]\in
{\go a}_{p-1}$. An easy computation shows that $b_{D_1}({\bf
a})=b_Y({\bf a})(b_D({\bf a})-b_D({\bf a}-1))$. The same argument as
before proves that $\dim {\go a}_p=+\infty$.

\end{proof}

\vskip 10pt
%%%%%%%%%
 \begin{rem}\label{rem-igusa} \footnote {We   thank Thierry
Levasseur for  pointing out the result of Igusa. }   The fact that $\dim {\go
a}=+\infty$ (but not the assertion concerning $\dim {\go a}_p$) is also a consequence
of a result by Igusa   ([Ig]). Our proof is different from  his.
\end{rem}
%%%%%%%%%%%
\begin{prop}\label{a0-commutative}The Lie algebra ${\go a}_0$, as well as the
associative algebras ${\cal A}_0$ and ${\cal T}_0$ are commutative.
\end{prop}
 
\begin{proof}

      It is enough to prove the result for ${\cal T}_0$. Let $D_1$
and
$D_2$ be two operators in ${\cal T}_0$ with Bernstein-Sato polynomials
$b_{D_1}$ and
$b_{D_2}$ respectively. An easy calculation shows that for any ${\bf a}$,
$b_{[D_1,D_2]}({\bf a})=b_{D_1}({\bf a})b_{D_2}({\bf a})-b_{D_2}({\bf
a})b_{D_1}({\bf a})=0$. Therefore $[D_1,D_2]=0$.

 \end{proof}
 
 \vskip 10pt \vskip 10pt
 %%%%%%%%%%%%%%%%%%%%%%%%%%%%%%%%%%%%%%%%%%
%%%%%%%%%%%%%%%%%%%%%%%%%%%%%%%%%%%%%%%%%%%
\section{  Embeddings of ${\go a}$, $ {\cal   A}$ and $ {\cal   T }$ into ${ \bb
C} [t,t^{-1},t{d\over dt}]\otimes { \bb C} [X_1,\cdots,X_n]$.}
%%%%%%%%%%%%%%%%%%%%%%%%%%%%%%%%%%%%%%%%%%%%%%
 \vskip 10pt
\subsection{  Vector space embedding of ${ \bb C} [t,t^{-1}  ]\otimes
{ \bb C} [X_0,\cdots,X_n]$   into $End({ \bb C} [\Omega^+])$}\hfill

 \noindent   Recall from (3-1-3) the following multiplicity free decomposition of ${ \bb C}
[\Omega^+]$ into irreducible representations of $G$:

$$ { \bb C} [\Omega^+]=\bigoplus_{{\bf a}\in { \bb Z}\times{ \bb N} ^n}V_{{\bf a}}.$$

\vskip 5pt
  \noindent    If $P\in V_{{\bf a}}$, then 
\begin{align*}{ XP &=\Delta_0P \in V_{{\bf a}+1}&(4-1-1)\cr
 YP&=b_Y({\bf a}){P\over {\Delta_0}}=b_Y({\bf a})X^{-1}P \in V_{{\bf a}-1}&(4-1-2)\cr
EP&=b_E({\bf a})P\in V_{\bf a}&(4-1-3)}\end{align*}

\noindent where we know from   $(3-1-4)$ that
$$b_Y({\bf  X})=\prod _{j=0}^n(X_0+\cdots+X_j+j{d\over2}) $$
and where
$$b_E({\bf X})=(n+1)X_0+nX_1+\ldots+X_n\eqno (4-1-4).$$
Therefore the operators $E,X,Y,X^{-1}$ are very well understood as elements of
$End({ \bb C} [\Omega^+])$. More precisely the operator $E$ is diagonalisable on  ${ \bb
C} [\Omega^+]$, the eigenspaces are the spaces $V_{\bf a}$ and $E$ acts on each $V_{\bf a}$
by multiplication by the value at ${\bf a}$ of the polynomial $b_E $.  Similarly the action of
$Y$ "goes down" from $V_{\bf a}$ to $V_{ {\bf a}-1 }$. More precisely the action of $Y$ on
$V_{\bf a}$ is the action of $X^{-1}$ (division by $\Delta_0$) together with the
multiplication by the value of the polynomial  $b_Y$ at ${\bf a}$.

\vskip 5pt
	    \noindent   The preceeding remarks suggest that there is an embedding of the vector space ${ \bb C} [t,t^{-1}  ]\otimes
{ \bb C} [X_0,\cdots,X_n]$ into $End({ \bb C} [\Omega^+])$. Let us make this more precise.

\noindent       If
$U$ and
$V$ are vector spaces over ${ \bb C} $ we will denote by $ {\cal   L}(U,V)$ the vector space of
linear maps from $U$ to $V$.

      \noindent  Let us also remark that from the preceeding decomposition of ${ \bb C}
[\Omega^+]$ we get:
$$End({ \bb C} [\Omega^+])=\bigoplus_{{\bf a}\in { \bb Z}\times { \bb N} ^n} {\cal  
L}(V_{\bf a},{ \bb C} [\Omega^+]) \eqno(4-1-5)$$

%%%%%%%%%%
 \begin{definition}\label{def-nouveau-produit}

For ${\bf a}\in { \bb Z}\times { \bb N} ^n$, $P\in { \bb C} [X_1,\ldots,X_n], \,  q\in
{ \bb C} [t,t^{-1}]$ we will denote by $\varphi_{\bf a}(q\otimes P)$ the element of $ {\cal  
L}(V_{\bf a},{ \bb C} [\Omega^+])$ defined by:
$$   Q_{\bf a}\in
V_{\bf a}, \quad \varphi_{\bf a}(q\otimes P)Q_{\bf a}=P({\bf a})q(\Delta_0)Q_{\bf a}$$ 
\noindent ($\varphi_{\bf a}$ defines a linear  map from $${ \bb C} [t,t^{-1}  ]\otimes { \bb C}
[X_0,\cdots,X_n]\text{  into  }{\cal   L}( V_{\bf a}, End({ \bb C} [\Omega^+])) $$
by the universal 
property of tensor products).

\noindent  We then define, using   $(4-1-5)$, an element  
$$\varphi\in  {\cal   L}({ \bb C} [t,t^{-1} 
]\otimes { \bb C} [X_0,\cdots,X_n], End({ \bb C} [\Omega^+]))   \text {  by  }$$
 $$\varphi=\bigoplus_{{\bf a}\in { \bb Z}\times { \bb N} ^n}\varphi_{\bf a}.$$
\end{definition}
\vskip 5pt
%%%%%%%
\begin{prop}\label{prop-applic-injective}

The linear map :
$$\varphi:{ \bb C} [t,t^{-1}  ]\otimes { \bb C} [X_0,\cdots,X_n]\longrightarrow
End({ \bb C} [\Omega^+])$$
is injective and its image is stable under  the multiplication (composition of mappings) in $End(
{ \bb C} [\Omega^+])$.
 \end{prop}
\begin{proof}
Any element $u$ in ${ \bb C} [t,t^{-1}  ]\otimes { \bb C} [X_0,\cdots,X_n]$ can be
written as a finite sum
$$u=\sum t^{i}\otimes P_i\qquad i\in { \bb Z},\,P_i\in { \bb C} [X_0,\cdots,X_n].$$
Suppose that $u\in  \ker\varphi$. Then for any $Q_{\bf a}\in V_{\bf a}$ we have:
$$0=\varphi(u)Q_{\bf a}=\sum \varphi(t^i \otimes P_i)Q_{\bf a}=\sum  P_i({\bf
a})\Delta_0^iQ_{\bf a}.$$
This implies that $P_i({\bf a})=0$ for all $i$ and all $\bf a$, and therefore $\varphi$
is injective.

   \noindent   To prove the stability under multiplication it suffices to prove that for any
$R,S\in { \bb C} [X_0,\cdots,X_n]$ and $\ell,m\in { \bb Z}$ the endomorphism
$\varphi(t^m\otimes R)\varphi(t^\ell \otimes S)$ belongs to the image of $\varphi$. If 
$Q_{\bf a}\in V_{\bf a}$, we have $\varphi(t^m\otimes R)\varphi(t^\ell \otimes S)Q_{\bf
a}=\varphi(t^m\otimes R)S({\bf a})\Delta_0^\ell Q_{\bf a} = R({\bf a}+\ell)S({\bf
a})\Delta_0^{m+\ell}Q_{\bf a}=\varphi(t^{m+\ell}\otimes   (\tau_{-\ell}R) S )Q_{\bf a}$ where
$\tau_{ \ell}R(X)=R(X-\ell)$ and where $X-\ell=(X_0-\ell,X_1,\ldots,X_n)$. This proves that
the image is stable under multiplication.

 \end{proof}
 \vskip 10pt
 %%%%%%%%%%%%%%%%%%%%%%%%%%%%%%%%%%%%%%%%%%%%%%%%
\subsection{  Algebra embedding of ${ \bb C} [t,t^{-1},t{d\over dt}]\otimes{ \bb C}
[X_1,\cdots,X_n]$ into $End({ \bb C} [\Omega ^+])$}\hfill

 \noindent   Recall that if ${\bf A}$ and ${\bf B}$ are two associative algebras the tensor
product ${\bf A}\otimes{\bf B}$ is again an algebra, called the tensor product algebra, with
the multiplication defined by
$$a_1,a_2\in {\bf A},b_1,b_2\in {\bf B}\qquad (a_1\otimes b_1)(a_2\otimes
b_2)=a_1a_2\otimes b_1b_2 \eqno (4-2-1)$$
The tensor product algebra ${ \bb C} [t,t^{-1}  ]\otimes{ \bb C}
[X_0,\cdots,X_n] $ is commutative. On the other hand the algebra  $End({ \bb C} [\Omega
^+]) $ is of course non commutative. The linear injection $\varphi$ defined in the preceeding
section is therefore not an algebra homomorphism. But we will use $\varphi$ to define a new
multiplication in  ${ \bb C} [t,t^{-1}  ]\otimes{ \bb C} [X_0,\cdots,X_n] $ by setting   for  $q,r\in { \bb C} [t,t^{-1} ]$   and for $P,Q\in { \bb C} [X_0,\cdots,X_n]$:   
 $$(q\otimes P)(r\otimes Q)=\varphi^{-1}(\varphi(q\otimes P)\varphi(r\otimes Q))\eqno (4-2-2)$$ 
 More explicitly it is easy to see that for  $m,l\in { \bb Z}$ and $P,Q\in { \bb C} [X_0,\cdots,X_n]$ we have:
 $$  (t^m\otimes P)(t^\ell\otimes
Q)=t^{m+\ell}\otimes (\tau_{-\ell}P)Q \eqno (4-2-3)$$

  \noindent   With this multiplication ${ \bb C} [t,t^{-1}  ]\otimes{ \bb C}
[X_0,\cdots,X_n] $ becomes a non commutative associative algebra.

 \noindent      We will denote by ${ \bb C} [t,t^{-1},t{d\over dt}]={ \bb C}
[t,t^{-1}, {d\over dt}]$ the algebra of differential operators on ${ \bb C} ^*$ whose
coefficients are Laurent polynomials. In fact, using the notation introduced at the beginning
of the section $3$, ${ \bb C} [t,t^{-1},t{d\over dt}]$ is  ${\bf
D}({ \bb C} ^*)$, the   Weyl algebra of the torus.

%%%%%%%
\begin{prop}\label{prop-extended-Weyl}

The algebra  ${ \bb C} [t,t^{-1}  ]\otimes{ \bb C} [X_0,\cdots,X_n] $ whose
multiplication is defined by  $(4-2-3)$  is isomorphic to the tensor product algebra
${ \bb C} [t,t^{-1},t{d\over dt}]\otimes{ \bb C}
[X_1,\cdots,X_n]$ (extended Weyl algebra of the torus).
 \end{prop}
 \begin{proof}
We define first an isomorphism 
$$\nu:{ \bb C} [t,t^{-1} ]\otimes{ \bb C}
[X_0,\cdots,X_n]\longrightarrow{ \bb C} [t,t^{-1},t{d\over dt}]\otimes{ \bb C}
[X_1,\cdots,X_n]$$
between the underlying vector spaces. After that we will prove that $\nu$ is in fact an
algebra isomorphism.

  \noindent  A vector basis of ${ \bb C} [t,t^{-1} ]\otimes{ \bb C}
[X_0,\ldots,X_n]$ is given by the elements $t^m\otimes X_0^{\alpha_0}X_1^{\alpha_1}\ldots
X_n^{\alpha_n}$ ($m\in { \bb Z},\alpha_i\in { \bb N} $). Let us define a linear map:
$$\nu(t^m\otimes X_0^{\alpha_0}X_1^{\alpha_1}\ldots X_n^{\alpha_n})=t^m\big(t{d\over
dt}\big)^{\alpha_0}\otimes  X_1^{\alpha_1}\ldots
X_n^{\alpha_n}\eqno (4-2-4)$$
On the other hand the elements $t^m(t{d\over dt})^{\alpha_0}\otimes X_1^{\alpha_1}\ldots
X_n^{\alpha_n}$ define a vector basis of ${ \bb C} [t,t^{-1},t{d\over dt}]\otimes{ \bb C}
[X_1,\cdots,X_n]$. Therefore the linear map $\nu$~is effectively a vector space isomorphism.

 \noindent  Let us recall that if ${\bf X}=(X_0,\ldots,X_n)$, then 
$\widetilde{\bf X}=(X_1,\ldots,X_n)$. In order to prove that $\nu$ is an isomorphism of
algebras, it is enough  to prove that 
$$\nu[(t^m\otimes X_0^iA({\bf \widetilde{X}}))(t^\ell \otimes X_0^j      B({\bf
\widetilde{X}}))]=\nu [ t^m\otimes X_0^iA({\bf \widetilde{X}}) ]\nu [ t^\ell      \otimes 
X_0^jB({\bf
\widetilde{X}}) ]\eqno (4-2-5)$$
for any $m, \ell \in { \bb Z}, i, j\in { \bb N}$ and for any $A, B \in { \bb C} [X_1,
\ldots,X_n]$, where the first product is defined by $(4-2-3)$ in ${ \bb C} [t,t^{-1}
]\otimes{ \bb C} [X_0,\ldots,X_n]$ and the second one is the usual product in the tensor
algebra ${ \bb C} [t,t^{-1},t{d\over dt}]\otimes{ \bb C} [X_1,\cdots,X_n]$.
\vskip 5pt
 \noindent  We will need the following lemma whose easy proof by induction is left to the reader.

%%%%%%%
\begin{lemma}\label{lemme-technique}

For any $i,j\in { \bb N} $,  $\ell\in { \bb Z}$:
$$(t{d\over dt})^it^\ell (t{d\over dt})^j=\sum_{p=0}^i {i\choose p}\ell^{i-p}t^\ell (t{d\over
dt})^{p+j}.$$
 \end{lemma}
\vskip 5pt
 \noindent     We have then:
 \vskip 5pt

$\nu[(t^m\otimes X_0^iA({\bf \widetilde{X}}))(t^\ell     \otimes X_0^j  B({\bf
\widetilde{X}}))]=\nu[ t^{m+\ell}\otimes (X_0+\ell)^iX_0^jA({\bf \widetilde{X}})   B({\bf
\widetilde{X}}) ]$

$=\displaystyle{ \sum_{p=0}^i }\nu(t^{m+\ell}\otimes  {i\choose p}\ell^{i-p}X_0^{p+j} A({\bf
\widetilde{X}})   B({\bf
\widetilde{X}))}$

$ = \displaystyle{\sum_{p=0}^i  t^{m+\ell}{i\choose
p}\ell^{i-p} (t{d\over dt})^{p+j}\otimes A({\bf
\widetilde{X}})   B({\bf
\widetilde{X})}}$
\vskip 5pt

   \noindent  On the other hand we have:
   
   \vskip 5pt

\noindent $\nu [ t^m\otimes X_0^iA({\bf \widetilde{X}}) ]\nu [ t^\ell      \otimes 
X_0^jB({\bf
\widetilde{X}}) ]= t^m(t{d\over dt})^it^\ell(t{d\over dt})^j\otimes A({\bf \widetilde{X}})   B({\bf
\widetilde {X}})$ and then using Lemma \ref{lemme-technique} we get:
\vskip 5pt

 $\displaystyle{\sum_{p=0}^i } t^{m+\ell}{i\choose
p}\ell^{i-p} (t{d\over dt})^{p+j}\otimes A({\bf
\widetilde{X}})   B({\bf \widetilde{X}})$.

\end{proof}
\vskip  5pt

   \noindent   The algebra $ { \bb C} [t,t^{-1},t{d\over dt}]\otimes{ \bb C}
[X_1,\cdots,X_n] $ can be viewed as the algebra of differential operators on the torus ${ \bb C}
^*$ with coefficients in ${ \bb C}   [X_1,\cdots,X_n]$.   In other words any element of this
algebra can be written as a finite sum of the form: 
$$\sum A_{r,s}({\bf \widetilde{X}})t^r(t{d\over dt})^s\quad {\rm or}\quad \sum B_{r,s}({\bf
\widetilde{X}})t^r({d\over dt})^s  \eqno(4-2-6)$$
where     $r \in { \bb Z}$, $s \in { \bb N}$ and $A_{r,s},B_{r,s}\in { \bb C}
[X_1,\cdots,X_n]$.

	    \noindent   Relations $(4-1-1),(4-1-2)$ and $(4-1-3)$ imply then easily the following Corollary.  
%%%%%%%%%
\begin{cor}\label{cor-realisationA}
The Lie algebra ${\go a}$ is
isomorphic to the Lie subalgebra of
$ { \bb C} [t,t^{-1},t{d\over dt}]\otimes{ \bb C} [X_1,\cdots,X_n] $ generated by

 $b_E(t{d\over dt},X_1,\ldots,X_n)=(n+1)t{d\over dt}+nX_1+\cdots+X_n$, 
 
 $t$ and $t^{-1}b_Y(t{t\over dt},X_1,\ldots,X_n)=t^{-1}\prod_{j=0}^n(t{d\over dt}+X_1+\cdots+X_j+j{d\over
2}).$

 \noindent Similarly  $ {\cal   A}$ is isomorphic to the associative subalgebra generated by these generators
and $ {\cal   T   }$ is isomorphic to the associative subalgebra generated by $b_E(t{d\over
dt},X_1,\ldots,X_n)$, $t$, $t^{-1}b_Y(t{t\over dt},X_1,\ldots,X_n)$ and $t^{-1}$.
\end{cor}

%%%%%%%%
\begin{definition}\label{def-comp-rad}

An element $D\in {\bf D}(\Omega^+)$ is said to have a radial component if there exists an operator
${\rm r}_D\in {\bf D}({ \bb C} ^*)$ such that for any
$f\in { \bb C} [{ \bb C} ^*]$  one has $D(f\circ\Delta_0)={\rm r}_D(f)\circ \Delta_0$. The
operator ${\rm r}_D$ is then called the radial component of $D$.
\end{definition}

 %%%%%%%%
\begin{prop}\label{prop-comp-rad }

The operators $E$, $X$, $X^{-1}$ and $Y$ have radial components which are given respectively by

$$\begin{array}{l} {\rm r}_E=t{d\over dt},\,{\rm r}_X=t,\,{\rm r}_{X^{-1}}= {1\over t} \\  
\\
 {\rm r}_Y={1\over t}b_Y(t{d\over dt},0,\ldots,0)={1\over t}\prod_{j=0}^n(t{d\over dt}+j{d\over 2}) 
 \end{array}$$
\end{prop}
\begin{proof}

   Only the formula for ${\rm r}_Y$ needs a proof.  The Proposition \ref{prop-extended-Weyl} gives an algebra embedding 
$$\Psi=\varphi\circ \nu^{-1}:  { \bb C} [t,t^{-1},t{d\over dt}]\otimes{ \bb C}
[X_1,\cdots,X_n] \longrightarrow End({ \bb C} [\Omega^+]).$$
 It is easy to see that if $A\in
{ \bb C} [X_1,\cdots,X_n]$, $m\in { \bb Z}$, $n\in { \bb N} $,  then the operator
$\Psi(A({\bf \widetilde{X}})t^m(t{d\over dt})^n)$ can be described as follows: 

\noindent for $a_0\in { \bb Z}$ and $ Q_{\widetilde{\bf a}} \in V_{\widetilde{\bf a}}$, we have 
$$   \Psi(A({\bf \widetilde{X}})t^m(t{d\over dt})^n)\Delta_0^{a_0} Q_{\widetilde{\bf a}}= A({\widetilde{\bf
a}})a_0^n\Delta_0^{a_0+m}Q_{\widetilde{\bf a}} =A({\widetilde{\bf
a}})[t^m(t{d\over dt})^n (t^{a_0})\circ \Delta_0]Q_{\widetilde{\bf a}}  \eqno (4-2-7)$$

 \noindent From Corollary \ref{cor-realisationA} one gets that $Y=\Psi(t^{-1}b_Y(t{d\over dt},X_1,\ldots,X_n))$. Therefore, as
$\Delta_0\in V_{(m,0,\ldots,0)}$, we deduce from $(4-2-7)$ that 

$$Y\Delta_0^m=[{1\over t}b_Y(t{d\over
dt},0,\ldots,0)(t^m)]\circ \Delta_0.$$

\noindent This gives ${\rm r}_Y={1\over t}\prod_{j=0}^n(t{d\over dt}+j{d\over 2})$.

 \end{proof}
%%%%%%%
\begin{rem} \label{rem-comp-rad-Rais}
 In the $A_{2k-1}$ case (notation as  in Table 1) the relative
invariant $\Delta_0$ is the  determinant of a $k\times k$ matrix, and $Y= \det(\partial)$. The
radial component has been calculated by Ra\"is ([Ra], page 22). He obtained that ${\rm
r}_Y=\prod_{j=2}^k(t{d\over dt}+j){d\over dt}$, whereas the preceeding formula  leads to ${\rm
r}_Y={1\over t}\prod_{j=0}^{k-1}(t{d\over dt}+j )$ (in that case ${d\over 2}=1$). A simple computation
shows that these two differential operators are the same.
\end{rem}
%%%%%%%%%%%%%%%%%%%%%%%%%%%%%%%%%%%%%%%%%%%%
\vskip 10pt \vskip 10pt
\section{ Several  Algebras of Invariant Differential  Operators. }

 %%%%%%%%%%%%%%%%%%%%%%%%%%%%%%%%%%%%%%%%%%%%%
\subsection{   A first result}\hfill
 
\vskip 5pt
      \noindent We are in a situation where two different groups,   $G$ and $G'$, act on two
affine varieties, namely $\Omega^+$ and $ V^+$. This gives rise to the following four algebras of
invariant differential operators:

$${\bf D}(\Omega^+)^G, \quad {\bf D}(\Omega^+)^{G'}, \quad {\bf D}( V^+)^{G }, \quad {\bf
D}(V^+)^{G'}.$$

%%%%%%%%
\begin{prop}\label{prop-regularite} Every $G$-invariant differential operator on $\Omega^+$ has polynomial
coefficients, i.e. ${\bf D}(\Omega^+)^G= {\bf D}( V^+)^{G } $.
\end{prop}
 
\begin{proof}This result is well known, even in the $C^{\infty}$ context (see for example [Y], Remark 2, and [No]).
It is usually proved by exhibiting a set of generators of ${\bf D}(\Omega^+)^G$ which have
polynomial coefficients. We give here a direct proof. Recall from $(3-1-2) $ and $(3-1-3)$ the
decompositions of ${ \bb C} [V^+]$ and ${ \bb C} [\Omega^+]$ into irreducible representations
of $G$:
$$ { \bb C} [V^+]=\bigoplus_{{\bf a} \in { \bb N} ^{n+1}} V_{\bf
a}\quad,\qquad { \bb C} [\Omega^+]=\bigoplus_{{\bf a} \in { \bb Z}\times{ \bb
N} ^{n }} V_{\bf a}$$
If $D\in {\bf D}(\Omega^+)^G$, then by Schur's Lemma $D$ maps  each $V_{\bf a}$ into
itself. Therefore such a $D$ stabilizes ${ \bb C} [V^+]$. This implies that $D\in {\bf D}(
V^+)^{G }$.

\end{proof}

     \noindent Among the preceding  spaces the following inclusions are obvious:
$$\begin{matrix}{\bf D}(V^+)^G & = &{\bf D}(\Omega^+)^G\\
\downarrow & { }&\downarrow \\
{\bf D}(V^+)^{G'} & \longrightarrow &  {\bf D}(\Omega^+)^{G'}\end{matrix}$$
 In this chapter we will give several descriptions of these algebras. 
\vskip 10pt
 
%%%%%%%%%%%%%%%%%%%%%%%%%%%%%%%%%%%%%%%%%%%%%
 \subsection{ The Harish-Chandra isomorphism for $G/H   $  and a first description 
  of ${\bf D}(\Omega^+)^G \simeq {\bf D}( G/H)^G$} \hfill

  \noindent   As $\Omega^+\simeq G/H$ is a complex symmetric space it is well known   that ${\bf
D}(\Omega^+)^G
\simeq {\bf D}( G/H)^G$ is isomorphic to a polynomial algebra through the so-called Harish-Chandra
isomorphism.

 \noindent    For the convenience of the reader let us first recall some details  of the
Harish-Chandra isomorphism for $G/H$. In fact what is needed here is an algebraic version
of this isomorphism because our algebras of differential operators are defined algebraically. It 
can be easily deduced from the "real analytic case" given in   [H-S], Theorem 4.3. part II.
\vskip 10pt

 \noindent    Let ${\go q}$ be the orthogonal complement of ${\go h}$ in ${\go g}$  with respect
to the Killing form of $  {\go{g}}  $ . Therefore one has $ {\go{g}} ={\go h}\oplus  {\go q}$. It is
known ([B-R], Chap. 5)  that ${\go t}=\sum_{i=0}^n{ \bb C} H_{\alpha_i}$ is a maximal abelian
subspace of ${\go q}$.  

 \noindent    Let $\widetilde{\Sigma}=\Sigma( \widetilde {\go{g}} ,{\go t})$ be the set of roots of $(
\widetilde {\go{g}} ,{\go t})$ and  let
$\Sigma =
\Sigma( {\go{g}} ,{\go t})$ be the set of roots of   $( {\go{g}} ,{\go t})$. Of course the roots in
$\widetilde{ \Sigma}$ are just the restrictions to ${\go t}$ of the roots of $\widetilde{     R}$. It is
possible to define an order on
$\Sigma $ such that if
$\alpha\in  {    R^+}$, then the restriction $\overline{\alpha}=\alpha_{|_{\go t}} $ belongs to   $\Sigma
^+$ (see [B-R], Chapter 5,  for example). The Weyl group of $\Sigma$ is denoted by $W$. Let 

$${\go n}^+=\sum_{\gamma\in \Sigma ^+ }  {\go{g}}^\gamma,  \quad {\go n}^-=\sum_{\gamma\in \Sigma ^- } 
{\go{g}}^\gamma\eqno(5-2-1)$$

\noindent where, as usual, ${\go{g}}^\gamma$ denotes the root space  corresponding to $\gamma$. Then we have
the decomposition 
$$ {\go{g}} = {\go h}\oplus{\go t} \oplus {\go n}^- \eqno (5-2-2)$$
and therefore, using the Poincar\'e-Birkhoff-Witt  Theorem, the universal enveloping algebra $ {\cal
  U}({\go g})$ of $ {\go{g}} $ decomposes as follows:
$$ {\cal U}({\go g})= S({\go t})\oplus( {\cal U}({\go g}){\go h}+{\go n}^- {\cal U}({\go g}))
\eqno(5-2-3)$$
where $S({\go t})$ is the symmetric algebra of ${\go t}$. Let us denote by $\gamma'$ the projection
from ${\cal U}({\go g})$ onto $S({\go t})$ defined by the decomposition  $(5-2-3)$.

 \noindent   Let ${\cal U}({\go g})^{{\go h}}$ be the space of elements in ${\cal U}({\go g})$ which
commute with ${\go h}$. It is known that $\gamma'$ is a surjective algebra homomorphism from ${\cal
U}({\go g})^{{\go h}}$ onto $S({\go t})$ ([D] 7.4.3.) .

\noindent     The space $S({\go t})$ is canonically identified with the space of polynomial functions
on ${\go t}^*$. Let $\rho={1\over2}\sum_{\lambda\in \Sigma^-}\lambda $ and define for any
$\Lambda\in {\go t}^*$ and any $z\in {\cal U}({\go g})$:

$$\gamma(z)(\Lambda)=\gamma'(z)(\Lambda-\rho) \eqno ( 5-2-4)$$

  \noindent   The map $\gamma$ factorizes through ${\cal U}({\go g}){\go h}\cap{\cal U}({\go
g})^{{\go h}}$ and defines an isomorphism of algebras (still denoted by $\gamma$):
$$\gamma:{\cal U}({\go g})^{{\go h}}/{\cal U}({\go g}){\go h}\cap{\cal U}({\go g})^{{\go
h}}\longrightarrow S({\go t})^W\eqno (5-2-5)$$
where $S({\go t})^W$ is the algebra of $W$ invariants in $S({\go t}) $.
 
  \noindent  This isomorphism is called the {\it Harish-Chandra isomorphism}.

\vskip 5pt
   \noindent  On the other hand, if $X\in   {\go{g}} $ and $\varphi \in { \bb C} [G]$, one
defines an element $\widetilde{X}\in  {\bf D}(G)^{\ell G}$ (the left invariant differential operators on $G$) by:
$$ \forall g \in G \quad \widetilde{X}\varphi(g)={d\over dt}\varphi(g\exp tX)_{|_{t=0}} \eqno (5-2-6)
$$
The map $X\mapsto\widetilde{X}$ extends to a map $U \mapsto \widetilde{U}$ from
${\cal U}({\go g})$ to ${\bf D}(G)^{\ell G}$.

  \noindent   For $f\in { \bb C} [G/H]={ \bb C} [\Omega^+]$ put $\widetilde{f}=f\circ\pi$
where $\pi:G\rightarrow G/H$ is the canonical projection. The map $f\mapsto \widetilde{f}$ is
then an isomorphism from ${ \bb C} [G/H]$ onto ${ \bb C} [G]^H$. Let us denote by $\varphi
\mapsto \overline{\varphi}$ the inverse mapping from ${ \bb C} [G]^H$ onto ${ \bb C} [G/H]$
given for $g\in G $ by $\overline{\varphi}(\overline{g})=\varphi(g)$.

 \noindent    It is easy to verify that if $U\in {\cal U}({\go g})^{{\go h}}$ and if $f\in { \bb
C} [G/H]$, then $\widetilde{U}\widetilde{f}\in { \bb C} [G]^H$.

\noindent     For $U\in {\cal U}({\go g})^{{\go h}}$ and $f\in { \bb C} [G/H]$, define $D_U\in
{\bf D}(G/H)   $ by:

$$(D_Uf)(\overline
{g})=\overline{\widetilde{U}\widetilde{f}}(\overline{g})=\widetilde{U}\widetilde{f}(g)$$
Then it is easy to see that $D_U\in {\bf D}(G/H)^G$. Let us call $r$ this map
$U\mapsto D_U$. It is well known that the map $r$ again factorizes through ${\cal U}({\go g}){\go h}\cap{\cal U}({\go
g})^{{\go h}}$ and defines an isomorphism of algebras (still denoted by $r$):
$$r:{\cal U}({\go g})^{{\go h}}/{\cal U}({\go g}){\go h}\cap{\cal U}({\go g})^{{\go
h}}\longrightarrow {\bf D}(G/H)^G \eqno (5-2-7)$$
 From $(5-2-5)$ and $(5-2-7)$ we deduce the following theorem (which is sometimes also
called the {\it Harish-Chandra Isomorphism }  ([H-S], Th. 4.3)):

%%%%%%%%% 
\begin{theorem}\label{th-iso-HC}
 The map $\gamma:{\bf D}(G/H)^G\longrightarrow S({\go t})^W$ defined for $\dot U \in
{\cal U}({\go g})^{{\go h}}/{\cal U}({\go g}){\go h}\cap{\cal U}({\go g})^{{\go
h}}$ by
$\gamma (r(\dot U))=\gamma (\dot U)$ (where  $\gamma (\dot U)$ has been defined in $(5-2-5)$) is an   isomorphism of algebras.
\end{theorem}

\vskip 5pt
     \noindent The next proposition gives a way to compute the image of a given element $D\in {\bf
D}(G/H)^G$ under the Harish-Chandra isomorphism. Put ${\go p}^-=  {\go t}\oplus {\go n }^-$. Let $\Lambda$ be a character of the group $R_{{\go p}^-}=\exp
 {\go t}.\exp {\go n^-}\subseteq G$. We will also denote by $\Lambda$ the corresponding infinitesimal character on  ${\go t}$. Let $f_\Lambda\in { \bb C} [G/H]$ be a dominant vector with
weight $\Lambda$. This means  that:
$$\forall b\in R_{{\go p}^-},\,\forall \dot g \in G/H\quad f_\Lambda(b\dot g)=\Lambda(b)f_\Lambda(\dot g)
\eqno (5-2-8)$$
For  $D= r(U)\in {\bf D}(G/H)^G$ define $\gamma'(D)=\gamma'(U)$.  
%%%%%%%%%%%
\begin{prop}\label{prop-image-HC}We have:

 \centerline{$\forall D   \in {\bf D}(G/H)^G,\,\forall \dot g \in G/H\,\, Df_\Lambda(\dot
g)=\gamma'(D)(\Lambda)f_\Lambda(\dot g)= \gamma(D)(\Lambda+\rho)f_\Lambda(\dot g)$.} \end{prop}

\begin{proof} Although the preceding result is already known in different forms, we give a short proof for the convenience of the reader. As $R_{{\go p}^-}H$ is open in $G$, it is enough to prove that
$$\forall b  \in R_{{\go p}^-}, \forall U\in {\cal U}({\go g})^{{\go h}}, \text{   }D_{U}f_{\Lambda}(\dot b)=\gamma'(U)(\Lambda)f_{\Lambda}(\dot b).$$
As $b=na$ with $a\in T=\exp{\go t}$ and $n\in N^-= \exp{\go n}^-$ and as both $D_{U}$ and  $f_{\Lambda}$ are left invariant under $N^-$, it is enough to prove that
$$\forall a  \in A,   \text{   }D_{U}f_{\Lambda}(\dot a)=\gamma'(U)(\Lambda)f_{\Lambda}(\dot a).$$
Let $U\in {\cal U}({\go g})^{{\go h}}$. Then from $(5-2-3)$ we can write $U=\gamma'(U)+D_{1}X+YD_{2}$ where $\gamma'(U)\in S({\go t})$, $X\in {\go h}$,   $Y\in {\go n}^-$and $D_{1}, D_{2}\in {\cal U}({\go g})$. Using the invariance properties of $\widetilde{f_{\Lambda}}$ it is easy to see that $\widetilde{X}\widetilde{f_{\Lambda}}=0$ and $\widetilde{YD_{2}}\widetilde{f_{\Lambda}}(\dot a)=0$. Therefore we have $D_{U}f_{\Lambda}(\dot a)=\widetilde{U}\widetilde{f_{\Lambda}}(a)= \widetilde{\gamma'(U)}\widetilde{f_{\Lambda}}( a)$, for all $a\in T$. But for $X_{0}\in {\go t}$, it is almost obvious that $\widetilde{X_{0}}\widetilde{f_{\Lambda}}(a)=X_{0}(\Lambda)\widetilde{f_{\Lambda}}(a) $. The proposition follows.

\end{proof}

\vskip 5pt
    \noindent For $\ell=0,\ldots,n$ define the differential operators $D_\ell$ by 

$$D_\ell=\Delta_0^{1-\ell}\Delta_0(\partial)\Delta_0^\ell =X^{1-\ell}YX^\ell \eqno (5-2-9)$$
These differential operators were probably first considered by A. Selberg in the case of
cones of classical types (see [Ter]). They were also used by Yan [Y] who proves the  analogue of the
following Theorem for symmetric cones. These results on symmetric cones can also be found in  the book
by J. Faraut and A. Koranyi ([F-K]).

%%%%%%%%%%
\begin{theorem}\label{th-generateurs}
The operators $D_0,\ldots,D_n$ are algebraically independent generators of ${\bf
D}(\Omega^+)^G$. In other words:
$${\bf D}(\Omega^+)^G={ \bb C} [D_0,\ldots,D_n].$$
\end{theorem}
\vskip 5pt

   \begin{proof}
     First of all it is easy to see that the operators $D_\ell$ are
$G$-invariant (and therefore they have polynomial coefficients by Proposition \ref{prop-regularite}.  Using now the Harish-Chandra
isomorphism $\gamma$ from Theorem \ref{th-iso-HC}  it is enough to prove that the elements
$\gamma(D_0),\gamma(D_1),\ldots,\gamma(D_n)$ are linearly independent generators of $S({\go t})^W$. We first
need the following lemma.

%%%%%%%%
\begin{lemma} \label{lemme-systeme-Cn}  The root system $\Sigma( \widetilde {\go{g}} ,{\go t})$ is always of type $C_{n+1}$ and
the root sytem $\Sigma( {\go{g}} ,{\go t})=\Sigma$ is always of type $A_n$.
\end{lemma}
\vskip 5pt
   \begin{proof} Define 

$$\widetilde{E}_{ij} (k,\ell)=\{X\in  \widetilde {\go{g}} \,|\,[H_{\alpha_i},X]=kX,\,[H_{\alpha_j},X]=\ell
X,\,[H_{\alpha_p},X]=0\, {\rm if}\,p\not =i,j\} $$

 \noindent and
$$\widetilde{E}_i (k  )=\{X\in  \widetilde {\go{g}} \,|\,[H_{\alpha_i},X]=kX,\,[H_{\alpha_p},X]=0\, {\rm
if}\,p\not =k\}.$$
We know from [M-R-S, Lemme 4.1.] that:
$$V^+=\bigoplus_{i<j}\widetilde{E}_{ij} \oplus\bigoplus_{i=0}^n \widetilde{E}_i (2  )\eqno (5-2-10)$$  
and
 $$ {\go{g}} =\bigoplus_{i<j}\widetilde{E}_{ij}(-1,1) \oplus\bigoplus_{i=0}^n \widetilde{E}_i (0 
)\oplus\bigoplus_{i<j}\widetilde{E}_{ij}(1,-1)\eqno(5-2-11)$$ 

 \noindent    Moreover one has $\bigoplus_{i=0}^n \widetilde{E}_i (0 )={\go z}_ {\go{g}} ({\go t})$. The
preceding decompositions show that the spaces $\widetilde{E}_{ij}(1,1)$ and $ \widetilde{E}_i (2  )$ are the
root spaces of the pair ($ \widetilde {\go{g}} ,{\go t}$). Let $\varepsilon_i={1\over2}\alpha_i$ be the dual
basis of the basis $H_{\alpha_i}$. Now it is clear that the positive roots of $\Sigma( \widetilde {\go{g}}
,{\go t})$ are the linear forms $\varepsilon_i+\varepsilon_j $ ($i<j$), $\varepsilon_i-\varepsilon_j$ ($i<j$)
 and $2\varepsilon_i$. This characterizes the root systems C$_{n+1}$ and $A_n$.

 \end{proof}

%%%%
{ \bf End of the proof of Theorem \ref{th-generateurs}}:
\vskip 5pt

   \noindent  As a corollary the Weyl group $W$ of $\Sigma$ is isomorphic to the symmetric group S$_{n+1}$ of
$n+1$ variables acting by permutations on the coordinates with respect to the $\alpha_i $'s. In order to
compute
$\gamma(D_\ell)$ we need to know the highest weight of   the
$  {\go{g}}
$-module 
$V_{\bf a}$ with respect to ${\go p}^-$. This highest weight has been computed in [R-S-2](Lemme 3.8 p. 155).
The result is as follows (recall that ${\bf a}=(a_0,\ldots,a_n)$):
$$\Lambda({\bf a})=
-a_0\overline{\alpha_0}-(a_0+a_1)\overline{\alpha_1}-\ldots-(a_0+\cdots+a_n)\overline{\alpha_n} \eqno
(5-2-12)$$
(where $\overline{\alpha_i}$  denotes the restriction  of $\alpha_i$ to ${\go t}$).

   \noindent  It is now convenient to make the following change of variables:

 \centerline{$r_i=\sum_{\ell=0}^ia_\ell$ ($i=0,\ldots,n$) and ${\bf r}=(r_0,r_1,\ldots,r_n)$.}

  \noindent   Note that  
$${\bf a}\in { \bb Z}\times { \bb N} ^n \Leftrightarrow{\bf r}=(r_0,r_1,\ldots,r_n)\in
{ \bb Z}^n \quad {\rm and} \quad r_0\leq r_1\leq r_2\leq \cdots\leq r_n .$$
Let us write $\Lambda({\bf r})$ instead of $\Lambda({\bf a})$. Then
$$\Lambda({\bf r})=-\sum_{i=0}^nr_i\overline{\alpha_i}\eqno (5-2-13)$$
 We need also to compute $\rho$. This computation again has already been made in [R-S-2], (Lemme
3.9. p. 155). The result is the following:
$$\rho={d\over4}\sum_{i<j}(\overline{\alpha_i}-\overline{\alpha_j})={d\over4}\sum_{i=0}^n(n-2i)\overline{\alpha_i}
\eqno (5-2-14)
$$
On the other hand a simple computation shows that:
$$ 
b_{D_\ell}(s_0,\ldots,s_n)=b_Y(s_0+\ell,s_1,\ldots,s_n)=\prod_{i=0}^n(s_0+\ell+s_1+\ldots+s_i+i{d\over2}
)\eqno (5-2-15)$$
From $(5-2-15)$, $(5-2-14)$ and Prop. \ref{prop-image-HC} we get 
$$\gamma'(D_\ell)(\Lambda({\bf r}))=b_{D_\ell}({\bf
a})=\prod_{i=0}^n(r_i +\ell+i{d\over2})\eqno (5-2-16)$$
Hence, from $(5-2-13)$ , $(5-2-14)$ and $(5-2-16)$:

  $$\begin{array}{rl}
\gamma(D_\ell)(\Lambda({\bf r}))&= \gamma'(D_\ell)(\Lambda({\bf r})-\rho)=\gamma'(D_\ell)(-\displaystyle\sum_{i=0}^nr_i\overline{\alpha_i}-
{d\over4}\sum_{i=0}^n(n-2i)\overline{\alpha_i})\\
&\\
& =\displaystyle \gamma'(D_\ell)(\sum_{i=0}^n
(-r_i-{d\over4}n+{d\over2}i)\overline{\alpha_i})\\
&\\
& =\displaystyle\prod _{i=0}^n(r_i+{d\over4}n+\ell).\hfill (5-2-17)  \end{array}$$
 As expected by Theorem \ref{th-generateurs} the polynomials $\gamma(D_\ell)(r_0,\ldots,r_n)=\prod
_{i=0}^n(r_i+{d\over4}n+\ell)$ are invariant under $S_{n+1}$. Moreover it is easy to prove (and well known)
that these polynomials are algebraically independent generators of the algebra of symmetric polynomials. Thus 
Theorem 4.2.3. is proved.   

\end{proof}

\noindent Let us note the following corollary of the proof.
%%%%%%%%%%%%%
\begin{cor}\label{cor-lien-gamma-b}
For any $D\in {\cal T}_{0}$, let $b_{D}({\bf r})$ be the polynomial  in the ${\bf r}$ variable defined by $b_{D}$. Let $\rho=\frac{d}{4}(-n,-n+2,\dots,n)$. Then $$\gamma(D)({\bf r})=b_{D}({\bf r}-\rho).$$
\end{cor}
%%%%%%
\begin{rem} \label{rem-structure-T}  As $E\in {\bf D}(\Omega^+)^G$, we deduce from Theorem \ref{th-generateurs} that $ {\cal T} 
={ \bb C} [E,X,X^{-1},Y]={ \bb C} [ X,X^{-1},Y]$ ($E$ is already a polynomial in $X,X^{-1},Y$). In fact
this is also a consequence of {\rm Th\'eor\`eme 1.1.} in {\rm [R-S-2]}. 
\end{rem}

\vskip 10pt

%%%%%%%%%%%%%%%%%%%%%%%%%%%%%%%%%%%%%%%%%%%%%%%%%%
\subsection{  Connections with $ {\cal   T}$ and $ {\cal   T}_0$}\hfill

     We will obtain in this section various descriptions of $ {\cal   T}$ and $ {\cal   T}_0$ in
terms of invariant differential operators, and also a characterization of ${\bf D}(V^+)^{G'}$ (see Theorem
\ref{th-G'-invariants} below).
%%%%%%%%%
\begin{theorem} \label{th-egalite-algebres} \hfill

1 - $ {\cal   T}_0={\bf D}(V^+)^G= {\bf D}(\Omega^+)^G \simeq {\bf D}( G/H)^G$

2 -  ${\cal   T}={\bf D}(\Omega^+)^{G'}$.
\end{theorem}

 \begin{proof} 
 
 The equality ${\bf D}(V^+)^G= {\bf D}(\Omega^+)^G $ has already been proved in Prop.
4.1.1. It is clear from  relations $(3-1-1)$
       that the operators in $ {\cal   T}_0$ are
$ G$-invariant. Therefore ${\cal   T}_0\subseteq{\bf D}(\Omega^+)^G$. The converse inclusion is a consequence
of Theorem 4.2.3. Thus the first assertion is proved.

     Let $A\in {\bf D}(\Omega^+)^{G'}$. Then $A=\sum_{i\in { \bb Z}}A_i$ (finite sum), where
$[E,A_i]=iA_i$. This means just thet $A_i$ has global degree $i$.  As
$E$ is $G'$-invariant, it is clear that each $A_i$ is also $G'$-invariant. Let ${\bf a}\in {\bb N}^{n+1}$ be such that the
restriction of $A_i$ to $V_{\bf a}$ is non zero. From the $G'$-invariance we deduce that   there 
exists
$\ell\in { \bb Z}$ such that   the  operator $A_i$ maps $V_{\bf a}$ into $\oplus _{\ell \in {\bb Z}}V_{{\bf a}+\ell}$, and then,
for degree reasons,  the components $A_i$ are equal to zero unless $i=(n+1)\ell,\, \ell\in { \bb Z}$ and 
$A=\sum_{i\in { \bb Z}}A_{i(n+1)}$. Then the operator $\Delta_0^{-i}A_{i(n+1)}=X^{-i}A_{i(n+1)}$ is
$G'$-invariant and verifies
$[E,A_{i(n+1)}]=0$. As the Euler operator is the infinitesimal generator of the center of $G$, we obtain that
$X^{-i}A_{i(n+1)}\in {\bf D}(\Omega^+)^G={ \bb C} [D_0,\ldots,D_n]$ (Theorem \ref{th-generateurs}). Thus each
$A_{i(n+1)}$ is  a polynomial  in the operators of the form $X^mYX^p$, and hence $A\in { \bb
C}[E,X,X^{-1},Y]= {\cal   T}$.
 
 \end{proof}
\vskip 5pt
     As $V_{\bf a}$ is a $G$-irreducible module, it is well known that the tensor $G$ -module $V_{\bf
a}\otimes V_{\bf a}^*$ contains up to constant, a unique $G$-invariant vector $R_{\bf a}$ (see for example
[H-U]). Moreover as
$  { \bb C}[V^+]\otimes{ \bb C} [V^+]^*$ is $G$-isomorphic to ${\bf D}(V^+)$, the element $R_{\bf a}$
can be viewed as a $G$-invariant differential operator with polynomial coefficients. The operators $R_{\bf a}$
are sometimes called {\it Capelli Operators}. Moreover the family of elements  $R_{\bf a}$
 (${\bf a}\in { \bb N} ^{n+1}$)  is a vector basis of the vector space ${\bf D}(V^+)^G= {\bf
D}(\Omega^+)^G$
\vskip 10pt
     Another description of the algebra $  {\bf D}(\Omega^+)^G$, which is due to Zhimin Yan is as
follows.

%%%%%%%%
\begin{theorem} \label{th-description-yan}

 Let ${\bf 1}_j=(0,\ldots,0,1,0,\ldots,0)$ where $1$ is at the $j$-th place. 

 \noindent The elements $XY=R_{{\bf
1}_0},R_{{\bf
1}_1},R_{{\bf 1}_2},\ldots,R_{{\bf 1}_n} $ are algebraically independent generators of $ {\bf
D}(\Omega^+)^G$.  Therefore 
$$ {\cal   T}_0= {\bf D}(\Omega^+)^G= {\bf D}(V^+)^G={ \bb C} [XY,R_{{\bf
1}_1},R_{{\bf 1}_2},\ldots,R_{{\bf 1}_n}].$$ 
\end{theorem}

   \begin{proof} This is just a complex version of Theorem 1.9 of [Y]. 

\end{proof}
\vskip 5pt

     Let ${ \bb C} [X, Y, R_{{\bf
1}_1},R_{{\bf 1}_2},\ldots,R_{{\bf 1}_n} ] $ be the subalgebra of ${\bf D}(\Omega^+)$ generated by the
variables $X,Y,R_{{\bf
1}_1},R_{{\bf 1}_2},\ldots,R_{{\bf 1}_n}$. Note that the operators $X$ and $Y$ do not commute, and also
that they do not commute with $R_{{\bf
1}_1},R_{{\bf 1}_2},\ldots,R_{{\bf 1}_n}$.

%%%%%%%%%
\begin{theorem}\label{th-G'-invariants}

\noindent Denote   by ${\cal   T}_0[X,Y]$ the subalgebra of ${\bf
D}(\Omega^+)$ generated by $ {\cal   T}_0,\,X$ and $Y$. One has:
$${\bf D}(V^+)^{G'}={ \bb C} [X, Y, R_{{\bf
1}_1},R_{{\bf 1}_2},\ldots,R_{{\bf 1}_n} ]= {\cal   T}_0[X,Y]$$
\end{theorem}
 
   \begin{proof} The inclusion ${ \bb C} [X, Y,R_{{\bf
1}_1},R_{{\bf 1}_2},\ldots,R_{{\bf 1}_n}]\subset {\bf D}(V^+)^{G'}$ is obvious.

     Remember that the decomposition of ${ \bb C} [V^+]$ into
irreducible $G'$-modules is as follows:

$${ \bb C} [V^+]=\bigoplus_{\widetilde{\bf a}\in { \bb N}
^n}(\bigoplus_{a_0\geq 0}\Delta_0^{a_0}V_{\widetilde{\bf a}}) $$

     Here the various $\Delta_0^{a_0}V_{\widetilde{\bf a}}$ with the same
${\widetilde{\bf a}}$ are
$G'$-irreducible and $G'$-isomorphic, but if ${\widetilde{\bf a}}\not
={\widetilde{\bf b}}$, then the $G'$-modules $\Delta_0^{a_0}V_{\widetilde{\bf a}}$
and
$\Delta_0^{b_0}V_{\widetilde{\bf b}  }$ are non isomorphic. In other words the
space $ \bigoplus_{a_0\geq 0}\Delta_0^{a_0}V_{\widetilde{\bf a}}$ is the
isotypic component of the irreducible (harmonic) $G'$-module $V_{\widetilde{\bf
a}}= {\cal   H}_{\widetilde{\bf a}}$ (see $(3-1-5 )$).

The dual module of $\Delta_0^{a_0}V_{\widetilde{\bf a}}$ can naturally be identified with $\Delta_0^{-a_0}V_{\widetilde{\bf a}}^*\subset { \bb C} [V^+]^*$. Then we know (see the above definition of the $R_{\bf a}$'s) that the sub-module $\Delta_0^{i}V_{\widetilde{\bf a}}\otimes\Delta_0^{-j}V_{\widetilde{\bf a}}^*\subset D(V^+)^{G'}$ contains a unique $G'$-invariant element namely $X^iR_{\widetilde{\bf a}}Y^j$.

    Moreover  the set of elements $X^iR_{\widetilde{\bf a}}Y^j$
($i\geq 0,j\geq 0, {\widetilde{\bf a}} \in  { \bb N} ^n)$ is a vector basis of
the space ${\bf D}(V^+)^{G'}$. From the preceding theorem each
$R_{\widetilde{\bf a}}$ is a polynomial in $R_{{\bf 1}_0}=XY, R_{{\bf
1}_1},R_{{\bf 1}_2},\ldots,R_{{\bf 1}_n} $. Therefore    ${\bf D}(V^+)^{G'}\subset{ \bb C} [X, Y, R_{{\bf
1}_1},R_{{\bf 1}_2},\ldots,R_{{\bf 1}_n} ]$ and hence 
 $${\bf D}(V^+)^{G'}={ \bb C} [X, Y, R_{{\bf
1}_1},R_{{\bf 1}_2},\ldots,R_{{\bf 1}_n} ].$$

Obviously we have also the inclusion  ${\cal   T}_0[X,Y]\subset {\bf D}(V^+)^{G'}$. On the other hand, as the operators $R_{{\bf 1}_1},R_{{\bf 1}_2},\ldots,R_{{\bf 1}_n} $ are $G$-invariant, they belong to ${\cal   T}_0$. Therefore ${\bf D}(V^+)^{G'}={ \bb C} [X, Y, R_{{\bf
1}_1},R_{{\bf 1}_2},\ldots,R_{{\bf 1}_n} ]\subset {\cal   T}_0[X,Y]$. This completes the proof.

\end{proof}
%%%%%%%%
\begin{rem}\label{rem-cas-quadratique}

Note first that in all cases $R_{{\bf 1}_n}=E$. In the special case where
$G\simeq SO(k)\times{ \bb C} ^*$ and $V^+\simeq { \bb C} ^k$, the
preceding theorem yields
$${\bf D}({ \bb C} ^k)^{SO(k)}={ \bb C} [Q(x),Q(\partial),E]$$
where $Q(x)=X=\sum_{i=1}^kx_i^2$, $Q(\partial)=Y=\sum_{i=1}^k{\partial^2\over
\partial x_i^2}$.

   \noindent  This was   proved by S. Rallis and G. Schiffmann ([Ra-S], Lemma
5.2. p. 112).

\end{rem}

\vskip 10pt \vskip 10pt
 
%%%%%%%%%%%%%%%%%%%%%%%%%%%%%%%%%%%%%%%%%%%%%%%%%%%%%%%
\section {More structure}

%%%%%%%%%%%%%%%%%%%%%%%%%%%%%%%%%%%%%%%%%%%%%%%%%%%%%
\vskip 10pt
 \subsection{The automorphism $\tau$}\hfill
 %%%%%%
 \begin{definition}\label{def-tau}
 The automorphism $\tau$   of $ {\cal   T}$  is defined by:
$$\forall D\in  {\cal   T},\qquad \tau(D)=XDX^{-1} \eqno (6-1-1)$$
\end{definition}
%%%%%%%
\begin{prop} \label{prop-commutation-tau} 
The algebra $ {\cal   T}_0$ is stable under $\tau$ and for any $R\in {\cal   T}_0$ one has:
$$XR=\tau(R)X \eqno (6-1-2)$$
$$RY=Y\tau(R)\eqno (6-1-3)$$

\end{prop}

\begin{proof} If $D\in  {\cal   T}_0$, then for homogeneity reasons one has $\tau(D)=XDX^{-1}\in  {\cal   T}_0$.
As $XR=XRX^{-1}X$, the identity $(6-1-2)$ is obvious.

     We will now prove that $(6-1-3)$ holds on each subspace $V_{\bf a}$. Let $b_R$ the
Bernstein-Sato polynomial of $R$. Then an easy calculation shows that the left and right hand side  of
$(6-1-3)$ act on $V_{\bf a}$ as $b_R({\bf a}-1)b_Y({\bf a})X^{-1 }$.

\end{proof}

%%%%%%%%%
\begin{prop}\label{prop-description-T}\hfill

\noindent 1) Let $ {\cal   T}_0[X,X^{-1}] $  be the subalgebra of  ${\bf
D}(\Omega^+)$ generated by $ {\cal   T}_0$, $X$ and $X^{-1}$. One has 
$$ {\cal   T}= {\cal   T}_0[X,X^{-1}]. \eqno (6-1-4)$$

\noindent  More precisely any element  $D\in  {\cal   T}$ can be written uniquely in the form 
$$ 
D=\sum_{i\in {\bb Z}}u_i X^i   \text{  or  }  
 D=\sum_{i\in {\bb Z}}  X^i u_i \quad{\rm (finite \,\,sums)}  
\eqno(6-1-5)$$
with $u_i\in  {\cal   T}_0$. Therefore $ {\cal   T}$ is a free left and right $ {\cal   T}_0$-module.

\noindent 2) Any element $D$ in ${\cal   T}_0[X,Y] $ can be written uniquely in the form
$$D=\sum_{i> 0} u_iY^i    +  \sum_{i\geq 0} v_iX^i\text {  or  } D=\sum_{i> 0}  Y^i u_{i}   +  \sum_{i\geq 0}  X^iv_i\text {  {\rm (finite  sums)}}   \eqno (6-1-6)$$
with $u_i, v_{i}\in  {\cal   T}_0$. Therefore $ {\cal   T}_{0}[X,Y]$ is a free left and right $ {\cal   T}_0$-module.
\end{prop}

\begin{proof} \hfill

1) Recall from Remark \ref{rem-structure-T} that $ {\cal   T}={ \bb C} [X,X^{-1},Y]$. Therefore any $D\in  {\cal  
T}$ can be written as a linear combination of elements of the form $u=X^{i_1}Y^{k_1}\ldots
X^{i_\ell}Y^{k_\ell}$ where $i_j\in { \bb Z},k_j\in { \bb N}$ .  If $p=\sum_{j=1}^\ell i_j-k_j$,
then $u=X^{i_1}Y^{k_1}\ldots X^{i_\ell}Y^{k_\ell}X^{-p}X^p=mX^p$  and $m=X^{i_1}Y^{k_1}\ldots
X^{i_\ell}Y^{k_\ell}X^{-p}\in  {\cal   T}_0$.

     Suppose now that $\sum_i u_iX^i=0$ (finite sum, $u_i\in  {\cal   T}_0$). As $u_iX^i\in 
{\cal   T}_i$, this implies that for all $ i$, one has $u_iX^i=0$. As $ {\cal   T}$ is an integral domain,
we have $u_i=0$.

2) From the definition the elements of ${\cal   T}_0[X,Y] $ are sum of monomials $M= X^{i_{1}}R_{k_{1}}Y^{j_{1}}X^{i_{2}}R_{k_{2}}Y^{j_{2}}\dots X^{i_{k}}R_{k_{k}}Y^{j_{k}}$ where $i_{\ell}, j_{\ell}\geq 0$ and where  $R_{k_{\ell}}\in {\cal   T}_0$. From Proposition \ref{prop-commutation-tau}, we obtain that $M=X^{\alpha }R_{\beta }Y^{\gamma}$ with $\alpha, \gamma \in {\bb N}$ and $R_{\beta}\in {\cal   T}_0$. Suppose now that $M$ is homogeneous of degree $p\geq 0$, then $M=X^{\alpha }R_{\beta }Y^{\gamma}= uX^{p}$ where $u= X^{\alpha }R_{\beta }Y^{\gamma}X^{-p}  \in {\cal   T}_0$. If $M$  is homogeneous of degree $p< 0$, then $\gamma> -p$ and $M=u Y^{-p}$ where $u= X^{\alpha}R_{\beta}Y^{\gamma + p}\in {\cal   T}_0$. This shows the existence of the decomposition $(6-1-6)$. The uniqueness of the decomposition is proved as in 1).

\end{proof}
\noindent  The following corollary is then obvious.
%%%%%%%%%%%
\begin{cor}\label{cor-parties-positives}

The inclusion  ${\cal   T}_0[X,Y]={\bf D}(V^+)^{G'}\subset  {\cal   T}_0[X,X^{-1}]={\bf
D}(\Omega^+)^{G'}= {\cal   T}$ is strict but these two algebras have the same "positive" part: for all
$p\geq 0$ one has $ {\cal   T}_p={\cal   T}_0[X,X^{-1}]_p= {\cal   T}_0X^p= {\cal   T}_0[X,Y]_p$.
\end{cor}

\vskip 10pt
 %%%%%%%%%%%%%%%%%%%%%%%%%%%%%%%%%%%%%%%
 \subsection {The center of $ {\cal   T}$.}  \hfill
 
 \vskip 5pt

     \noindent As in section 5.2., let $  {\go{g}} ={\go h}\oplus{\go q}$ where ${\go h}$ is the Lie
algebra of the generic isotropy group $H$ and ${\go q}$ is the orthogonal of ${\go h}$ with respect to
the Killing form. Let $S_{n+1}$ be the symmetric group in $n+1$ variables $r_0,r_1,\ldots,r_n$ and let
${ \bb C} [r_0,r_1,\ldots,r_n]^{S_{n+1}}$ be the algebra of symmetric polynomials.
\vskip 10pt

     Recall from section 5.2. that $S_{n+1}$ is the Weyl group of the root system
$\Sigma=\Sigma({\go g},{\go t})$ where ${\go t}=\sum_{i=0}^n{ \bb C} H_{\alpha_i}$ is a maximal
abelian subspace of ${\go q}$. Recall also that we have made the change of variables ${\bf
a}\longleftrightarrow {\bf r}$ where ${\bf r}=(r_0,r_1,\ldots,r_n) $ is defined by $r_i=
\sum_{\ell=0}^ia_\ell$. Then the highest weight $\Lambda({\bf r}) $ of $V_{\bf a}$ with respect to ${\go
p}^-$ is  given by $\Lambda({\bf r})=-\sum_{i=0}^n r_i\overline{\alpha_i}$ where $\overline{\alpha_i}$
denotes the restriction of $\alpha_i$ to ${\go t}$ (see $(5-2-13)$). More generally any element ${\bf
r}=(r_0,r_1,\ldots,r_n)\in { \bb C} ^{n+1}$ can be identified with the element $\Lambda({\bf
r})=-\sum_{i=0}^n r_i\overline{\alpha_i}\in {\go t}^*$. We make then, for any $D\in  {\cal  T }$,  the
convention that $b_D(\Lambda({\bf
r}))=b_D(r_0,r_1,\ldots,r_n)=b_{D}(a_{0},a_{1},\ldots,a_{n})$.

\noindent  Let $\rho$ be the half sum of roots in $\Sigma ^-$. Recall from $(5-2-14)$ that
$\rho={d\over 4}\sum_{i=0}^n (n-2i)\overline{\alpha_i}$. Then, as seen in section $5.2.$, the Harish-Chandra isomorphism $\gamma:{\bf D}(\Omega^+)^G= {\cal  
T}_0\longrightarrow { \bb C} [r_0,r_1,\ldots,r_n]^{S_{n+1}}$ is given by 

$$ \forall D\in  {\cal T}_0, \qquad \gamma(D)(r_0,r_1,\ldots,r_n)=b_D(\Lambda({\bf r})-\rho) \eqno
(6-2-1)$$
\vskip 5pt

\noindent  Denote also by $\tau$ the isomorphism of ${ \bb C} [r_0,r_1,\ldots,r_n]^{S_{n+1}}$ defined for $  P\in  { \bb C} [r_0,r_1,\ldots,r_n]^{S_{n+1}}$ by 

$$  
\tau(P)(r_0,r_1,\ldots,r_n)=P(r_0+1,r_1+1,\ldots,r_n+1)\eqno
(6-2-2)$$
\vskip 5pt

 \noindent    Let us  denote  by  ${ \bb C} [r_0,r_1,\ldots,r_n]^{S_{n+1},\tau}$ the algebra of
$\tau$-invariant symmetric polynomials.
\vskip 5pt

     Recall that we have previously also denoted by $\tau$ the conjugation by $X$ in $ {\cal   T}$
(see  $(6-1-1)$). An easy calculation shows that for $D\in  {\cal   T}_0$ we have
$b_{\tau^{-1}(D)}(a_0,a_1,\ldots,a_n)=b_{X^{-1}DX}(a_0,a_1,\ldots,a_n)=b_D(a_0+1,a_1,\ldots,a_n)$. In
the ${\bf r}$ variable this gives
$b_{\tau^{-1}(D)}(r_0,r_1,\ldots,r_n)=b_D(r_0+1,r_1+1,\ldots,r_n+1)=\tau(b_D)(r_0,r_1,\ldots,r_n)$.
Therefore the two definitions of $\tau$ are coherent in the sense  that they
correspond under the Harish-Chandra isomorphism. More precisely we have:

$$ \forall D\in  {\cal T}_0, \qquad  \gamma({\tau^{-1}(D)})({\bf r})=\tau(\gamma(D))({\bf r}) \eqno
(6-2-3)$$
%%%%%%%%%%
\begin{lemma}\label{lem-points-fixes-tau}

For an operator $D\in  {\cal   T}_0$ the following conditions are equivalent:

     i) $\tau(D)=D$ (i.e. $D$ commutes with $X$).

      ii) $\forall {\bf a}\in { \bb N} ^{n+1}$, $b_D({\bf a}+1)=b_D({\bf a}) $ where as
in section $3.1.$, ${\bf a}+1=(a_0+1,a_1,\ldots,a_n)$.

     iii) $\gamma(D)\in { \bb C} [r_0,r_1,\ldots,r_n]^{S_{n+1},\tau}$ (i.e. $\gamma (D)$
is $\tau$-invariant).
\end{lemma}
 
\begin{proof} This is just a consequence of $(6-2-3)$ and of the discussion before.

\end{proof}
%%%%%%%%%%
\begin{theorem}\label{th-centre-T}

Let $ {\cal   Z}( {\cal   T})$ be the center of $ {\cal   T}$.

   \noindent  1) Then $D\in {\cal   Z}( {\cal   T})$ if and only if $D\in  {\cal   T}_0$ and $D$
commutes with $X$ (i.e. $\tau(D)=D$).

\noindent    2) The center of ${\cal   T}$ is also the center of ${\cal   T}_{0 }[X,Y]$, i.e. $ {\cal   Z}( {\cal   T})=  {\cal   Z}( {\cal   T}_{0}[X,Y]) $

  \noindent   3)  ${\cal   Z}( {\cal   T})$ is also the set of elements $D$ in ${\cal   T}_0$ such that $b_{D}({\bf a})$ does not depend on $a_{0}$.

  \noindent  4) ${\cal   Z}( {\cal   T})=\gamma^{-1}({ \bb C}[r_0,r_1,\ldots,r_n]^{S_{n+1},\tau})$.

\end{theorem}

\vskip 10pt
\begin{proof}Let $D\in  {\cal   T}$. Using the ${ \bb Z}$-gradation we can write $D=\sum D_i$ (finite sum),
where $D\in  {\cal   T}_i$. Suppose now that $D\in {\cal   Z}( {\cal   T})$. Then
$[E,D]=0=\sum(n+1)iD_i$. Therefore $D_i=0$ if $i\not =0$, thus $D\in  {\cal   T}_0$. Moreover, of
course, $D$ commutes with $X$.

     \noindent Conversely suppose that $D\in  {\cal   T}_0$ and that $DX=XD$. Then from Prop. \ref{prop-description-T} we
obtain that $D$ commutes with every element in $ {\cal   T}$, i.e. $D\in {\cal   Z}( {\cal   T})$. The
first assertion is proved.

\noindent The second assertion is obvious.

  \noindent    The third and the fourth assertions are     consequences of the first one and of Lemma \ref{lem-points-fixes-tau}

\end{proof}

%%%%%%%%%%%
\begin{rem}\label{rem-commutant-X} As a consequence of the preceding theorem it can be noticed that an operator
$D\in  {\cal   T}_0$ which commutes with $X$, automatically commutes with $Y$.

\end{rem}

%%%%%%%%%%%
\begin{lemma}\label{lem-poly-translation} 

Let $ {\cal   M}$ be the hyperplane of ${ \bb C} ^{n+1}$ defined by 
$${\cal   M}=\{(r_0,r_1,\ldots,r_n)\in { \bb C} ^{n+1}|\, r_0+r_1+\cdots+r_n=0\}.$$ 
 Let $I({\cal   M})=
\{P\in { \bb C}
[r_0,r_1,\ldots,r_n]^{S_{n+1}}|\,P_{|_{\cal   M}}=0\}$.Then $$I({\cal   M})=(r_0+r_1+\cdots+r_n){ \bb C}
[r_0,r_1,\ldots,r_n]^{S_{n+1}}\text{  and  }$$  
$${ \bb C}
[r_0,r_1,\ldots,r_n]^{S_{n+1}}={ \bb C}
[r_0,r_1,\ldots,r_n]^{S_{n+1},\tau}\oplus I({\cal   M}) \eqno (6-2-4)$$
\end{lemma}

\begin{proof} Let $P\in I({\cal   M})$. As $ {\cal   M} $ is an irreducible hyperplane defined by the
irreducible polynomial $r_0+r_1+\cdots+r_n$, we have $P=(r_0+r_1+\cdots+r_n)Q$ where $Q\in  { \bb C}
[r_0,r_1,\cdots,r_n]$. As $P$ and $r_0+r_1+\cdots+r_n$ are $S_{n+1}$-invariant, the polynomial $Q$ is
also $S_{n+1}$-invariant. Hence 
$$I({\cal   M})=(r_0+r_1+\cdots+r_n){ \bb C}
[r_0+r_1+\cdots+r_n]^{S_{n+1}}.$$

  \noindent   Define $F={ \bb C} .(1,1,\ldots,1)$. Then ${ \bb C} ^{n+1}= {\cal   M}\oplus F$. Let
$Q\in { \bb C} [ {\cal   M}]^{S_{n+1}}$ be an $S_{n+1}$-invariant polynomial on ${\cal   M}$. Then
$Q$ can be extended to a polynomial $\widetilde{Q}$ on ${ \bb C} ^{n+1}$ by setting:
$$\forall m\in {\cal   M},\,\forall f\in F\,\,\quad \widetilde{Q}(m+f)=Q(m) \eqno(6-2-5).$$
Then $\widetilde{Q}\in { \bb C}
[r_0,r_1,\ldots,r_n]^{S_{n+1},\tau}$. In fact $P\longrightarrow P_{|_{\cal   M}}$ is a bijective map
from ${ \bb C}
[r_0,r_1,\ldots,r_n]^{S_{n+1},\tau}$ onto ${ \bb C} [ {\cal   M}]^{S_{n+1}}$, whose inverse map is
$Q\longrightarrow \widetilde{Q}$. Let $P\in { \bb C}
[r_0,r_1,\ldots,r_n]^{S_{n+1},\tau}\cap I( {\cal   M})$. Then $P_{|_{\cal   M}}=0$, and as $P$ is also
$\tau$-invariant  we get $P=0$. Hence ${ \bb C}
[r_0,r_1,\ldots,r_n]^{S_{n+1},\tau}\cap I( {\cal   M})=\{0\}$. 

  \noindent   On the other hand for any $P\in { \bb C}
[r_0,r_1,\ldots,r_n]^{S_{n+1} } $ we have $P=\widetilde{P_{|_{\cal   M}}}+ (P-\widetilde{P_{|_{\cal  
M}}})$. The polynomial $\widetilde{P_{|_{\cal   M}}}\in { \bb C}
[r_0,r_1,\ldots,r_n]^{S_{n+1},\tau}$ and the polynomial $(P-\widetilde{P_{|_{\cal  
M}}})$ vanishes on $ {\cal   M}$. Therefore $(P-\widetilde{P_{|_{\cal  
M}}})\in I( {\cal   M})$. This proves $(6-2-4)$.

\end{proof}

 \noindent An easy induction on the degree of $P$ leads to the following corollary.
 %%%%%%%%%%%%%%%
\begin{cor}\label{cor-decomp-pol-sym}
Let $P\in { \bb C}[r_0,r_1,\ldots,r_n]^{S_{n+1} }  $. Then $P$ can be uniquely wriiten in the form
$$P(r_{0},r_{1},\dots,r_{n})=\sum_{i=0}^p\alpha_{i}(r_{0},r_{1},\dots,r_{n})(r_{0}+\dots+r_{n})^p$$
where $\alpha_{i}\in { \bb C}[r_0,r_1,\ldots,r_n]^{S_{n+1},\tau}$.
\end{cor}

\vskip 10pt
%%%%%%%%%%%
\begin{prop}\label{prop-T0=centre+}

$$ {\cal   T}_0= {\cal   Z}( {\cal   T})\oplus E {\cal   T}_0 \eqno (6-2-6)$$
\end{prop}

\begin{proof} As before let $\gamma$ be the Harish-Chandra isomorphism between $ {\cal   T}_0$ and ${ \bb C}
[r_0,r_1,\ldots,r_n]^{S_{n+1} }$. As $b_E({\bf a})=(n+1)a_0+na_1+\cdots+a_n=r_0+r_1+\cdots+r_n$, one has
$\gamma (E)(r_0+r_1+\cdots+r_n)=b_E(\Lambda({\bf r})-\rho)=b_E(- \sum_{i=0}^n(r_i+{d\over
4}(n-2i) \overline{\alpha_i}))=\sum_{i=0}^n r_i \,(=b_{E}({\bf r}))$. Therefore, using Theorem \ref{th-centre-T}, the
decomposition $(6-2-6)$ is just the image under $\gamma^{-1}$ of the decomposition $(6-2-4)$.

\end{proof}

 %%%%%%%%%%%%%%%
\begin{cor}\label{cor-T0-poly-E}\hfill

\noindent 1) Let $H\in {\cal T}_{0}$. Then $H$ can be   uniquely written in the form:
$$H=H_{0}+EH_{1}+E^2H_{2}+\cdots+ E^kH_{k}\text{    where  } H_{k}\in {\cal   Z}( {\cal   T})$$

\noindent 2) Let $D\in {\cal T} $, then $D$ can be uniquely written in the form:
$$D=\sum_{k\in {\bb Z}, \ell \in {\bb N}}H_{k,\ell}E^{\ell}X^{k} \text{  or  } D=\sum_{k\in {\bb Z}, \ell \in {\bb N}}H_{k,\ell} X^{k}E^{\ell} \text{  \rm (finite \,\,sums)}$$
where $H_{k,\ell}\in {\cal   Z}( {\cal   T})$

\noindent 3)  Let $D\in {\cal T}_{0}[X,Y] $, then $D$ can be uniquely written in the form:
$$D=\sum_{k\in {\bb N}^*, \ell \in {\bb N}}H_{k,\ell}E^{\ell}Y^{k}+\sum_{r\in {\bb N}, s \in {\bb N}}H'_{r,s}E^{s}X^{r} \text{   {\rm (finite \,\,sum)}   or  }$$
$$D=\sum_{k\in {\bb N}^*, \ell \in {\bb N}}H_{k,\ell} Y^{k}E^{\ell}+\sum_{r\in {\bb N}, s \in {\bb N}}H'_{r,s} X^{r}E^{s}  \text{   {\rm (finite \,\,sum)}}$$
where $H_{k,\ell}, H'_{r,s}\in {\cal   Z}( {\cal   T})$

\end{cor}

\begin{proof} The first assertion is a direct consequence of Proposition \ref{prop-T0=centre+}. Assertions 2) and 3) are consequences of  1) and  Proposition \ref{prop-description-T}.

\end{proof}

%%%%%%%%%%%%
\begin{rem}\label{rem-non-algebre} It may  be noticed that $(\oplus_{i<0} {\cal   T}_i)\oplus E{\cal  
T}_0\oplus(\oplus_{i>0}{\cal   T}_i)$ is not a subalgebra of ${\cal   T}$. Indeed it was shown in [R-S-2]
that there exists an operator $\omega_X \in {\cal   T}_{-1}$ such that
$[\omega_X,X]={k\over2}+{2\over{n+1}}E$ and this operator does not belong to $E{\cal   T}_0$.
\end{rem}
\vskip 10pt

%%%%%%%%%%%%%%%%%%%%%%%%%%%%%%%%%%%%%%%%%%%%%%%% 
\subsection {Ideals of $ {\cal   T}$}\hfill
\vskip 10pt

     Let $J$ be a left (resp. right) ideal of $ {\cal   T}$. Then $J$ is said to be a graded
left (resp. right) ideal if $J=\oplus_{i\in { \bb Z}}J_i$ where $J_i=J\cap {\cal   T}_i$.

%%%%%%%%%%%
\begin{theorem}\label{th-ideaux-T}\hfill

   \noindent  1) Let $J$ be a graded left ideal of $ {\cal   T}$, then  $J=\oplus_{i\in { \bb
Z}}X^iJ_0$. Conversely if $J_0$ is any ideal of the (commutative) algebra $ {\cal   T}_0$, then
$J=\oplus_{i\in { \bb
Z}}X^iJ_0$ is a graded left ideal of $ {\cal   T}$.

   \noindent    2)  Let $J$ be a graded right  ideal of $ {\cal   T}$, then  $J=\oplus_{i\in { \bb
Z}} J_0X^i$. Conversely if $J_0$ is any ideal of the (commutative) algebra $ {\cal   T}_0$, then
$J=\oplus_{i\in { \bb
Z}} J_0X^i$ is a graded right  ideal of $ {\cal   T}$.

 \noindent      3) Let $J$ be a two-sided ideal of $ {\cal   T}$. Then $J$ is graded,   $J_0$ is a
$\tau$-invariant ideal of $ {\cal   T}_0$ and $J=\oplus_{i\in { \bb
Z}}X^iJ_0=  \oplus_{i\in { \bb
Z}} J_0 X^i$. Conversely  if $J_0$ is a $\tau$-invariant ideal of $ {\cal   T}_0$, then  $J=\oplus_{i\in { \bb
Z}}X^iJ_0=  \oplus_{i\in { \bb
Z}} J_0 X^i$ is a two-sided ideal of $ {\cal   T}$.
\end{theorem}

\begin{proof}{\,} 

Let $J=\oplus_{i\in { \bb Z}}J_i$ be a graded left ideal. Then $J_0$ is an ideal of $ {\cal  
T}_0$ and $X^iJ_0\subset J\cap  {\cal   T}_i=J_i$ and conversely $X^{-i}J_i\subset J\cap {\cal  
T}_0=J_0$. Therefore $J=\oplus_{i\in { \bb
Z}}X^iJ_0$.

   \noindent    If $J_0$ is an ideal of $ {\cal   T}_0$, define $J=\oplus_{i\in { \bb
Z}}X^iJ_0$. Let $D_j\in {\cal   T}_j$, then $D_jX^iJ_0=X^jX^i(X^{-i}X^{-j}D_jX^i)J_0$. As
$(X^{-i}X^{-j}D_jX^i)\in  {\cal   T}_0$, we obtain that $D_jX^iJ_0\in X^{i+j}J_0\subset J$. Hence $J$
is a graded left ideal.

   \noindent    The proof for graded right ideals is the same.

 \noindent      Let now $J$ be a two-sided ideal. An element $D\in J$ can be written uniquely
$D=\sum_{i=-\ell}^ \ell D_i$, where $D_i\in  {\cal   T}_i$. As $E\in {\cal   T}_0$, we have
$[E,D]=ED-DE\in J$. Moreover $[E,D]=\sum_{i=-\ell}^\ell[E,D_i]=\sum_{i=\ell}^\ell(n+1)iD_i\in
J$. By iterating the bracket with $E$ we get:
$$k=0,1,\ldots,2\ell \quad (\ad E)^kD=\sum_{i=-\ell}^\ell(n+1)^ki^kD_i\in J \eqno(6-3-1)$$
 \noindent     The square matrix defined by the linear system $(6-3-1)$ is invertible because its
determinant is Van der Monde. Therefore each operator $D_i$ belongs to $J$. Hence $J$ is graded.

 \noindent      As $J$ is a two-sided ideal, we have $XJ_0X^{-1}\subset J_0$, and this means that $J$ is
$\tau$-invariant.

   \noindent    Applying part 1) and part 2) we see that $J=\oplus_{i\in { \bb Z}}X^iJ_0=\oplus_{i\in
{ \bb Z}} J_0X^i$.

  \noindent     Conversely let $J_0$ be a $\tau$-invariant ideal of $ {\cal   T}_0$. Define $J=\oplus_{i\in { \bb
Z}}X^iJ_0$. According to 1), $J$ is a graded left ideal. But as $J_0$ is $\tau$-stable one has
$X^iJ_0X^{-i}=J_0$. Therefore $J=\oplus_{i\in { \bb
Z}}(X^iJ_0X^{-i })X^i= \oplus_{i\in { \bb
Z}} J_0X^i$. Then, according to 2), $J$ is also a graded right ideal, hence two-sided.

\end{proof}

   \noindent   The preceding theorem shows that the two-sided ideals of ${\cal T}$ are in one to one
correspondence with the $\tau$-invariant  ideals of ${\cal T}_0\simeq { \bb C}
[X_0,X_1,\ldots,X_n]^{S_{n+1}}$. For completeness we indicate  how such ideals are obtained.

%%%%%%%%%%%
\begin{prop}\label{prop-ideaux-tau}

 Any ideal  of ${ \bb C}
[X_0,X_1,\ldots,X_n]^{S_{n+1}}$ which is $\tau$-invariant  is generated by a finite
number of $\tau$-invariant polynomials.

\end{prop}
\vskip 10pt
\begin{proof}  Observe   that $\tau$ can be  defined as an automorphism of the algebra ${ \bb C}
[X_0,X_1,\ldots,X_n] $ by the same formula as in $(6-2-2)$.   Set
$Y_0=X_0,Y_1=X_1-X_0,Y_2=X_2-X_1,\ldots,Y_n=X_n-X_{n+1}$. Then for any polynomial $P$ we have
$(\tau P)(Y_0,\ldots,Y_n)=P(Y_0+1,Y_1,\ldots,Y_n)$. Therefore $P$ is $\tau$-invariant if and only
if it depends only on the variables $Y_1,\ldots Y_n$. Let
$P=P_d(Y_1,\ldots,Y_n)Y_0^d+P_{d-1}(Y_1,\ldots,Y_n)Y_0^{d-1}+\ldots+P_0(Y_1,\ldots,Y_n)$ be the
expansion of $P$ according to the powers of $Y_0$ in the ring ${ \bb C} [Y_1,\ldots,Y_n][Y_0]$.
An easy induction on the degree $d$ shows that $P_d$ is a linear combination of
$P,\tau P,\tau^2P,\ldots$. Then by induction the same is true for every coefficient
$P_i(Y_1,\ldots,Y_n)$ $(0\leq i \leq d)$. Note that the $P_i$'s are $\tau$-invariant.

    \noindent    Consider first an ideal $J$ of ${ \bb C}
[X_0,X_1,\ldots,X_n]$ which is (globally) $\tau$-invariant and suppose that $P\in J$. Then from
above, we deduce that $P_i\in J$  $(0\leq i \leq d)$.

    \noindent  Let now $J_0$ be a $\tau$-invariant ideal of ${ \bb C}
[X_0,X_1,\ldots,X_n]^{S_{n+1} }$ and put  $\sigma_0(X_0,\ldots,X_n) =X_0+X_1+\cdots+X_n$.
Suppose that $P\in J_0$. Then $P-P_d \sigma_0 ^d$ is still in $J_0$ and its degree in $X_0$ is
stricly less than $d$. By induction on the degree in $X_0$, we find that 
$$P=\sum P_jQ_j \eqno(6-3-2)$$
where every $P_j$ is in $J_0$ and is  $\tau$-invariant and where every $Q_j$ belongs to  ${ \bb
C} [X_0,X_1,\ldots,X_n]^{S_{n+1} }$. Hence the collection of all polynomials $P_j$ for all $P\in
J_0$ is a set of $\tau$-invariant generators for $J_0$. As  ${ \bb C}
[X_0,X_1,\ldots,X_n]^{S_{n+1} }$ is noetherian, there exists a finite set of such generators.

\end{proof}
 \vskip 10pt
%%%%%%%%%%%%%%%%%%%%%%%%%%%%%%%%%%%%%%%%%%%%
\subsection{  Noetherianity.}\hfill
\vskip 10pt

  \noindent     Recall that a non commutative ring $ {\cal   R}$ is said to be noetherian if  the right
and left ideals are finitely generated, or equivalently if the right and left ideals verify the
ascending chain condition (see for example [MC-R]).

 \noindent      Recall also that the rings ${\cal T}_0[X]$, ${\cal T}_0[X^{-1}]$,${\cal T}={\cal T}_0[X,X^{-1}]$,
${\cal T}_0[Y]$, ${\cal T}_0[X,Y]$ are defined to be the subrings of ${\bf D}(\Omega^+)$ generated by
${\cal T}_0$ and by the elements $X$,$X^{-1}$,$\{X,X^{-1}\}$, $Y$ and $\{X,Y\}$ respectively.

%%%%%%%%%%
\begin{theorem} \label{th-noether}The rings ${\cal T}_0[X]$, ${\cal T}_0[X^{-1}]$,${\cal T}={\cal T}_0[X,X^{-1}]$,
${\cal T}_0[Y]$, ${\cal T}_0[X,Y]$ are noetherian.
\end{theorem}
\vskip 10pt

\begin{proof}                   Let $S$ be a ring and $\sigma \in \text{Aut} S$. A $\sigma$-derivation of $S$ is a linear map $\delta: S\longrightarrow S $ such  that $\delta(st)=s\delta(t)+\delta(s)\sigma(t)$. Given a $\sigma$-derivation, the skew   polynomial ring determined by $\sigma$
 and $\delta $ is the ring $S[T, \sigma, \delta]:=S\langle T\rangle/\{sT-T\sigma(s)-\delta (s)|s\in S\}$ (see [MC-R], section 1.2  for details)

  \noindent Recall from Proposition \ref{prop-commutation-tau}  that for all $ R\in {\cal T}_0$ one has $XR=\tau(R)X,
YR=\tau^{-1}(R)Y$, where $\tau(R)=XRX^{-1}$. Recall also from Proposition \ref{prop-description-T} that any element $D\in
{\cal T}={\cal T}_0[X,X^{-1}]$ can be written uniquely $D=\sum_{i\in { \bb Z}}u_iX^i$, with $u_i\in
{\cal T}_0 $. The same easy argument shows that any element $D\in {\cal T}_0[Y]$ can be written
uniquely  $D=\sum_{i\in { \bb Z}}u_iY^i$ with $u_i\in {\cal T}_0$. 

   \noindent     These remarks imply that the rings ${\cal T}_0[X]$, ${\cal T}_0[X^{-1}]$ and ${\cal
T}_0[Y]$ (which are subrings of ${\bf D}(\Omega^+)$) are respectively isomorphic to the "abstract" skew
 polynomial rings ${\cal T}_0[T,\tau, 0]$ and ${\cal T}_0[T,\tau^{-1},0]$. They are therefore noetherian by Theorem 1.2.9. of [MC-R].

   \noindent    The ring ${\cal T}={\cal T}_0[X,X^{-1}]$ is similarly isomorphic to the skew  
algebra of Laurent polynomials ${\cal T}_0[T,T^{-1},\tau]$ and is therefore noetherian by Theorem
1.4.5. of [MC-R].

    \noindent   The relations $XR=\tau(R)X, YR=\tau^{-1}(R)Y$, where $R\in {\cal T}_0$, imply that
${\cal T}_0X+{\cal T}_0=X{\cal T}_0+{\cal T}_0$. Moreover $[X,Y]\in {\cal T}_0$. These remarks imply
that ${\cal T}_0[X,Y]$ is an almost normalizing extension of ${\cal T}_0$ in the sense of [MC-R]
(section 1.6.10.). As ${\cal T}_0$ is noetherian, this implies by Theorem 1.6.14. of [MC-R], that
${\cal T}_0[X,Y]$ is noetherian.

\end{proof}
\vskip 10pt

%%%%%%%%%%%%%%%%%%%%%%%%%%%%%%%%%%%%%%%%%%%%%%%%%%%%% 
\subsection{  Gelfand-Kirillov dimension.}\hfill
\vskip 10pt

  \noindent   We will denote by $GK.\dim( {\cal   R})$ the Gelfand-Kirillov dimension of the algebra
$ {\cal   R}$.

%%%%%%%%%%%
\begin{theorem}

One has $GK.\dim ({\cal T})=GK.\dim ({\cal T}_0[X])=GK.\dim ({\cal T}_0[X^{-1}])=GK.\dim ({\cal
T}_0[Y])=GK.\dim ({\cal T}_0[X,Y])=n+2.$
\end{theorem}

\begin{proof} We have seen in the proof of Theorem \ref{th-noether} that the algebras   ${\cal T}_0[X]$, ${\cal
T}_0[X^{-1}]$, ${\cal T}_0[Y]$ are isomorphic to the skew   polynomial algebra ${\cal
T}_0[T,\tau]$ (or ${\cal T}_0[T,\tau^{-1}]$) and that the algebra ${\cal T}_0[X,X^{-1}]$ is
isomorphic to the skew   algebra of Laurent polynomials ${\cal T}_0[T,T^{-1},\tau]$.

   \noindent    An automorphism $\nu$ of ${\cal T}_0$ is called locally algebraic if for any $D\in {\cal T}_0$,
the set $\{\nu^n(D),n\in { \bb N} \}$ spans a finite dimensional vector space. We know from [L-M-O] (Prop.1) that if $\tau$ is locally algebraic then
$GK.\dim({\cal T}_0[T,\tau])=GK.\dim({\cal T}_0[T,T^{-1},\tau])=GK.\dim({\cal T}_0)+1$ (see also
[Z]).  

  \noindent     Let us prove that $\tau$ is locally algebraic in the preceding  sense. The elements $D\in {\cal
T}_0$ are in one to one correspondence with their Bernstein-Sato polynomial $b_D$. An easy
computation shows that $b_{\tau(D)}({\bf a})=b_{XDX^{-1}}({\bf a} )=b_D({\bf a}-1)$. Therefore the
Bernstein-Sato polynomials $b_{\tau^n(D)}$ have the same degree as $b_D$. Hence the space spanned by
the family
$b_{\tau^n(D)}$ is finite dimensional, and $\tau$ is locally algebraic.

    \noindent   As ${\cal T}_0$ is a polynomial algebra in $n+1$ variables (see Theorem \ref{th-generateurs}) we
have  $GK.\dim ({\cal T}_0)=n+1$ (see for example [MC-R], Prop. 8.1.15. p. 282). Therefore
$GK.\dim({\cal T}_0[X])=GK.\dim({\cal T}_0[X^{-1}])=GK.\dim({\cal T}_0[Y])=GK.\dim({\cal
T}_0[X,X^{-1}])=n+2$.

  \noindent     As ${\cal T}_0[X]\subset {\cal T}_0[X,Y]\subset {\cal T}_0[X,X^{-1}]$, we have also
$GK.\dim({\cal T}_0[X,Y])=n+2$.

\end{proof}

\vskip 10pt
  
 %%%%%%%%%%%%%%%%%%%%%%%%%%%%%%%%%%%%%%%%%%%%%%%%%%%%%%%%
%%%%%%%%%%%%%%%%%%%%%%%%%%%%%%%%%%%%%%%%%%%%%%%%%%%%%%%
\vskip 10pt
\section{Generators and relations for ${\cal T}_{0}[X,Y]$.}
%%%%%%%%%%%%%%%%%%%%%%%%%%%%%%%%%%%%%%%%%%%%%%%%%%%%%%%%
\subsection{Generalized Smith algebras}\hfill
\hskip 10pt

  \noindent Let $\bf R$ be a commutative associative algebra over ${\bb C}$ with  unit element $1$ and without zero divisors. Let $f\in {\bf R}[t]$ be a polynomial in one variable with coefficients in ${\bf R}$ and let $n\in {\bb N}$. In any associative algebra $A$ the bracket of two elements $a$ and $b$ is defined by $[a,b]=ab-ba$.

\vskip 5pt 
%%%%%%%%%%%%%%%%% 
\begin{definition}\label{def.algebre-smith}  The (generalized) Smith algebra $S({\bf R},f,n)$ is the associative algebra over ${\bf R}$ with generators $x,y,e$ subject to the relations
$$[e,x]=(n+1)x, \quad [e,y]=-(n+1)y, \quad [y,x]=f(e).$$
\end{definition}
\vskip 5pt
%%%%%%%%%%%%%%%%
\begin{rem} \hfill

1) In the case ${\bf R}={\bb C}$ and $n=0$, the algebras  $S({\bb C}, f, 0)$ where introduced by Paul Smith in [Sm] as a class of algebras similar to the enveloping algebra ${\cal U}({\go sl}_{2}({\bb C}))$. He also developed a  very interesting representation theory for these algebras (see [Sm]).

2) One can prove, as in [Sm], that if the degree of $f$ is $1$, and if the leading coefficient is invertible in ${\bf R}$, then $S({\bf R},f,n)$ is isomorphic to the enveloping algebra ${\cal U}({\go sl}_{2}({\bf R}))$.

\end{rem}

%%%%%%%%%%%%%%%
\begin{prop}\label{prop.Smith=poly-twistee} Let ${\go b}$ the $2$-dimensional Lie algebra over  ${\bf R}$, with basis $\{\varepsilon, \alpha\}$ and relation $[\varepsilon,\alpha]=(n+1)\alpha$. Let ${\cal U}({\go b})$ the enveloping algebra of ${\go b}$. Define an automorphism $\sigma$ of ${\cal U}({\go b})$ by $\sigma(\alpha)=\alpha$ and $\sigma(\varepsilon)=\varepsilon-(n+1)$ and define also a $\sigma$-derivation $\delta$ of ${\cal U}({\go b})$ by $\delta(\alpha)=f(\varepsilon)$ and $\delta(\varepsilon)=0$. Then $S({\bf R},f,n)\simeq {\cal U}({\go b})[t,\sigma,\delta]$  {\rm (see the proof of Theorem \ref{th-noether} for the definition of the skew polynomial algebra ${\cal U}({\go b})[t,\sigma,\delta]$)}.
\end{prop}

\begin{proof} The proof is the same as the one  given by P. Smith ([Sm], Prop. 1.2.). It suffices to remark that the algebra ${\cal U}({\go b})[t,\sigma,\delta]$ is an algebra over ${\bf R}$ with generators  $\varepsilon, \alpha,t$ subject to the relations 
 \begin{align*} 
&[\varepsilon,\alpha]=(n+1)\alpha, \\
&\alpha t=t\sigma(\alpha)+\delta(\alpha) \text{ which is equivalent to }\alpha t=t\alpha+f(\varepsilon),\\
&\varepsilon t=t\sigma(\varepsilon)+\delta(\varepsilon) \text{ which is equivalent to }\varepsilon t=t(\varepsilon-(n+1))=t\varepsilon-(n+1)t.
 \end{align*} 
 Then the isomorphism $S({\bf R},f,n)\simeq {\cal U}({\go b})[t,\sigma,\delta]$ is given by $e\longmapsto \varepsilon$, $x\longmapsto \alpha$ and $y\longmapsto t$

\end{proof}

  \noindent  The following Corollary is also analogous to Corollary 1.3. in [Sm] and corresponds to a kind of Poincar\'e-Birkoff-Witt Theorem for $S({\bf R},f,n)$.
%%%%%%%%%%%%%%%%%%%
\begin{cor}\label{cor.smith-basis}\hfill 

  \noindent  $S({\bf R},f,n)$ is a noetherian domain with ${\bf R}$-basis $\{y^{i}x^{j}e^{k},i,j,k\in{\bb N}\}$  {  \rm (}or any similar family of ordered monomials obtained by permutation of the elements $(y,x, e)${\rm )}.
 
\end{cor}
\begin{proof}(compare with  [Sm], proof of corollary 1.3 p.288). We know from [C] or [MC-R] Th.1.2.9,   that as  ${\cal U}({\go b})$ is a  noetherian domain,  so is $S({\bf R},f,n)\simeq {\cal U}({\go b})[t,\sigma,\delta]$. Since 
\begin{align*} {\cal U}({\go b})[t,\sigma,\delta]&=  {\cal U}({\go b})\oplus  {\cal U}({\go b})t\oplus   {\cal U}({\go b})t^2\oplus  {\cal U}({\go b})t^3\oplus \dots\oplus {\cal U}({\go b})t^\ell\oplus \dots\\
&=  {\cal U}({\go b})\oplus  t{\cal U}({\go b}) \oplus   t^2{\cal U}({\go b}) \oplus  t^3{\cal U}({\go b}) \oplus \dots\oplus t^\ell{\cal U}({\go b}) \oplus \dots
\end{align*} 
(direct sums of ${\bf R}$-modules) and since the Poincar\'e-Birkhoff-Witt Theorem is still true for enveloping algebras of Lie algebras which are free over rings (see [Bou-1]),   the ordered monomials in $(y,x,e)$   beginning  or ending with $y$ form a basis of the algebra $S({\bf R},f,n)$ . To obtain the basis  $\{e^iy^jx^k\}$  or $\{ x^ky^je^i\}$ it suffices to replace the algebra ${\go b}$ by the algebra ${\go b}_{-}$  which is generated by $e$ and $y$.

\end{proof}
%%%%%%%%%%%%%%%%%
\begin{rem}\label{rem.graduation}The adjoint action of e $(e\longmapsto [e,u])$ on $ S({\bf R},f,n)$ is semi-simple and gives a decomposition of $S({\bf R},f,n)$ into weight spaces:
$$S({\bf R},f,n)=\oplus_{\nu\in {\bb Z}}S({\bf R},f,n)^{\nu}$$
where $S({\bf R},f,n)^{\nu}=\{u\in S({\bf R},f,n), [e,u]={\nu}(n+1) u\}$. As $[e,x^{j}y^{i} e^{k}]=(n+1)(j-i)y^{i}x^{j}e^{k}$, we obtain, using Corollary \ref{cor.smith-basis}, that  the ordered monomials of the form $ x^{i}y^{i}e^{k}$ form an  ${\bf R}$-basis for $S({\bf R},f,n)^{0}$. Moreover as $yx=xy+f(e)$, it is easy to sea that $S({\bf R},f,n)^{0}={\bf R}[xy,e]={\bf R}[yx,e]$, where ${\bf R}[xy,e]$ (resp. ${\bf R}[yx,e]$) denotes the ${\bf R}$-subalgebra generated by $xy$ (resp. $yx$) end $e$.
\end{rem}
\vskip 10pt
%%%%%%%%%%%%%
\begin{lemma}\label{lemma.f=u-u} There exists an element $u\in {\bf R}[t]$, which is unique up to 
addition of  an element of ${\bf R}$, such that 
 $$f(t)=u(t+n+1)-u(t)\eqno (7-1-1)$$
\end{lemma}

\begin{proof} Consider the ${\bf R}$-linear map $\Delta: {\bf R}[t]\longrightarrow {\bf R}[t]$ defined by $(\Delta P)(t)=P(t+n+1)-P(t)$. Define  $P_{1}[t]=\frac{1}{n+1}t$, then $(\Delta P_{1})(t)=1$ and therefore the elements of ${\bf R}$ belong to the image of $\Delta$. Suppose that the space ${\bf R}[t]^{k}$ of polynomials in ${\bf R}[t]$ of degree less than $k$ belong to the image of $\Delta$. Let $P_{k+2}(t)=t^{k+2}$, then $(\Delta P_{k+2})(t)=t^{k+1} \mod {\bf R}[t]^{k+1}$ and therefore we have proved by induction that $\Delta$ is surjective. Moreover we see that $\ker(\Delta)= {\bf R}$.

\end{proof}

The following result is  similar to Propostion 1.5 in [Sm], one has just to be careful when working over an arbitrary ring, rather than ${\bb C}$. It shows that, analogously to the enveloping algebra of ${\go sl}_{2}$, there is a Casimir-like element which generates the center of  $S({\bf R},f,n)$ over ${\bf R}$.

%%%%%%%%%%%%%%%%%%
\begin{prop}\label{prop.casimir} Let $u$ be as in the preceding Lemma. Define  
$$\Omega_{1}=xy-u(e)  \text{   and \, } \Omega_{2}=xy+yx-u(e+n+1)-u(e).$$
 Then $\Omega_{2}=2\Omega_{1}$ and the center of $S({\bf R},f,n)$ is ${\bf R}\Omega_{1}={\bf R}\Omega_{2}$.
 \end{prop} 
 
 \begin{proof}
 From the defining relations of $ S({\bf R},f,n) $, we have $[y,x]=yx-xy=f(e)=u(e+n+1)-u(e)$, hence $yx=xy+u(e+n+1)-u(e)$ and therefore $\Omega_{2}=2\Omega_{1}$.
 
  Let us now prove that $\Omega_{1}$ is central. As  $\Omega_{1}\in {\bf R}[xy,e]=S({\bf R},f,n)^0 $  (Remark \ref{rem.graduation}), we see that $\Omega_{1}$  commutes with $e$.

   From the   defining relations  of $ S({\bf R},f,n) $ we have also $[e,x]=ex-xe=(n+1)x$, hence $ex=x(e+n+1)$ and therefore, for any $k\in {\bb N}$, $e^k=x(t+n+1)^k$.
 
 This implies of course that for any polynomial $P\in {\bf R}[t]$ we have
 $$P(e)x=xP(e+n+1) \text{ or } P(e-n-1)x=xP(e). \eqno (7-1-2)$$
 Similarly one proves that
 $$P(e)y=yP(e-n-1) \text{ or }P(e+n+1)y=yP(e). \eqno(7-1-3)$$
 Let us now show that $\Omega_{1}$ commutes with $x$:
 \begin{align*}
 x\Omega_{1}&=x(xy-u(e))=x^2y-xu(e)=x(yx-u(e+n+1)+u(e))-xu(e)\\
 &=xyx-xu(e+n+1) = xyx-u(e)x \text { (using } (7-1-2))\\
 &=\Omega_{1}x.
  \end{align*}
  A similar calculation using $(7-1-3)$ shows that $\Omega_{1}$ commutes also with $y$. Hence $\Omega_{1}$ belongs to the center of $ S({\bf R},f,n) $.
  
  Let now $z$ be a central element of $ S({\bf R},f,n) $. Then $z\in S({\bf R},f,n)^0$. We have   $ S({\bf R},f,n)^0 = {\bf R}[xy,e]={\bf R}[\Omega_{1},e]$, and hence $z$ can be written as follows:
  $$z=\sum c_{i}(e)\Omega_{1}^{i} \qquad \text { (finite sum) }$$
  where $c_{i}(e)\in {\bf R}[e]$.
  
  We have:
  \begin{align*}
  0&=[z,x]=[\sum c_{i}(e)\Omega_{1}^{i},x]=\sum [c_{i}(e),x]\Omega_{1}^{i}\\
  &=\sum (c_{i}(e)x-xc_{i}(e))\Omega_{1}^i= \sum x(c_{i}(e+n+1)-c_{i}(e))\Omega_{1}^i \text{ (using  (7-1-2)) }\\
  &=x(\sum (c_{i}(e+n+1)-c_{i}(e))\Omega_{1}^i) 
 \end{align*}
 As the algebra   $S({\bf R},f,n)$ has no zero divisors  we get:
 $$\sum (c_{i}(e+n+1)-c_{i}(e))\Omega_{1}^i=0\eqno (*)$$

  Let us now remark that the elements $e^j\Omega^i$ ($i,j\in {\bb N}$) are free over ${\bf R}$. Suppose that we have 
  $$\sum _{i,j}\alpha_{i,j}e^j\Omega_{1}^i=0\qquad \text {  with } \alpha_{i,j}\in {\bf R}.$$
  As $\Omega_{1}=xy-u(e)$, we have 
  $$\Omega_{1}^i=x^iy^i \text {  modulo  monomials in  } e^kx^py^p \text { with } p<i.$$
 Therefore for all $i,j$, we have $\alpha_{i,j}=0$. Then from $(*)$ and from Corollary \ref{cor.smith-basis} above we obtain $c_{i}(e+n+1)-c_{i}(e)=0 , \text{  for all } i,j$, and hence $c_{i}(t+n+1)-c_{i}(t)=0 , \text{  for all } i,j$. From the proof of Lemma \ref{lemma.f=u-u}, we obtain that $c_{i}\in {\bf R}$, for all $i$.
 
 \end{proof}
 
 %%%%%%%%%%%%%%%%%%%%%%%%%%%%%%%%%%%%%%%%%%%%%%%%%%%%%%%%%%%%%
 \subsection{Some quotients of Generalized Smith Algebras}\hfill
 
 \vskip 10pt

  \noindent  Let $u\in {\bf R}[t]$ be an arbitrary polynomial with coefficients in ${\bf R}$, and let $n\in {\bb N}$.
 \vskip 10pt 
 %%%%%%%%%
\begin{definition}\label{def.algebre-quotient}  The   algebra $U({\bf R},u,n)$ is the associative algebra over ${\bf R}$ with generators $\tilde x,\tilde y,\tilde e$ subject to the relations
$$[\tilde e,\tilde x]=(n+1)\tilde x, \quad [\tilde e,\tilde y]=-(n+1)\tilde y, \quad \tilde x\tilde y=u(\tilde e), \quad \tilde y\tilde x=u(\tilde e +n+1) .$$
\end{definition}
\vskip 10pt
%%%%%%%%%%%%
\begin{rem}\label{rem-U=S/} Let $f\in {\bf R}[t]$ be  defined by  $f(t)=u(t+n+1)-u(t)$ (see Lemma \ref{lemma.f=u-u}). Then, from the definitions we have:
$$U({\bf R},u,n)= S({\bf R},f,n)/(  x  y-u(  e))=S({\bf R},f,n)/(\Omega_{1})$$
where $(  x  y-u(  e))$ is the   ideal (automatically two-sided) generated by $   x  y-u(  e) =\Omega_{1}$. Again, as for $S({\bf R},f,n)$,  the adjoint action of $\tilde{e}$ gives a decomposition of $U({\bf R},u,n)$ into weight spaces:
$$U({\bf R},u,n)=\oplus_{\nu\in {\bb Z}}U({\bf R},u,n)^{\nu}\eqno (7-2-1)$$
where $U({\bf R},u,n)^{\nu}=\{\tilde{v}\in U({\bf R},u,n), [\tilde{e},\tilde{v}]={\nu}(n+1) \tilde{v}\}$.
\end{rem} 
%%%%%%%%%%%%%
\begin{prop}\label{prop.U-contient-pol}Let $u\in {\bf R}[t]$ and $s\in {\bb N}$.The  ${\bf R}$-linear mappings
\begin{align*}
\varphi: {\bf R}[t] \longrightarrow  U({\bf R},u,n)&\quad &\psi: {\bf R}[t] \longrightarrow  U({\bf R},u,n)\\
 P \longmapsto  \varphi(P)=\tilde{x}^sP(\tilde{e})&\quad&P \longmapsto  \psi(P)=\tilde{y}^sP(\tilde{e})\\
 \end{align*} 
are injective   {\rm(}in particular the subalgebra ${\bf R}[\tilde{e}]\subset U({\bf R},u,n)$ generated by $\tilde{e}$ is a polynomial algebra{\rm)}.
 
\end{prop}

\begin{proof}\hfill

  \noindent Define $f(t)=u(t+n+1)-u(t)$.
Every element of $S({\bf R},f,n)$ can be written in a unique way under the form
 $$\sum a_{k,\ell,m}e^kx^\ell y^m \qquad (a_{k,\ell,m}\in {\bf R})$$
  (Corollary \ref{cor.smith-basis}). Therefore, from Remark \ref{rem-U=S/}, every element in $U({\bf R},u,n)$ can be written          under the form:
   $$\sum a_{k,\ell,m}\tilde{e}^k\tilde{x}^\ell \tilde{y}^m  \qquad (a_{k,\ell,m}\in {\bf R}).$$
  Let $P(t)=\sum_{i=0}^pa_{i}t^i$, $(a_{i}\in {\bf R})$ be a polynomial  such that $\tilde{x}^sP(\tilde{e})=0$  ({\it i.e.}      
    $ P\in \ker \varphi$). As $U({\bf R},u,n)=  S({\bf R},f,n)/(\Omega_{1})$ (Remark \ref{rem-U=S/}), we see that 
  $$x^s\sum_{i=0}^p a_{i}e^i\in (\Omega_{1})=\{r \Omega_{1}, r\in S({\bf R},f,n)\}.$$
  Therefore there exists $\alpha\in S({\bf R},f,n)$ such that 
  $$x^s\sum_{i=0}^p a_{i}e^i=\alpha \Omega_{1}=\alpha(xy-u(e)).$$
  If $\alpha=\sum a_{k,\ell,m}e^kx^\ell y^m $,  using the fact that $\Omega_{1}=xy-u(e)$ is central and relation $(7-1-2)$ we get:
 \begin{align*}\ x^s\sum_{i=0}^p a_{i}e^i &= (\sum_{k,\ell,m} a_{k,\ell,m}e^kx^\ell y^m) (xy-u(e))= \sum_{k,\ell,m} a_{k,\ell,m}e^kx^\ell(xy-u(e)) y^m \\
 &= \sum_{k,\ell,m} a_{k,\ell,m}e^kx^{\ell+1} y^{m+1}- \sum_{k,\ell,m} a_{k,\ell,m}e^kx^\ell u(e) y^m\\
 &=  \sum_{k,\ell,m} a_{k,\ell,m}e^kx^{\ell+1} y^{m+1}- \sum_{k,\ell,m} a_{k,\ell,m}e^ku(e-\ell(n+1))x^\ell   y^m   (**)
 \end{align*}
 Suppose now that $\alpha\neq 0$, then one can define 
 $$\ell_{0}=\max\{\ell \in {\bb N}, \exists k,m , \alpha_{k,\ell,m}\neq 0\}.$$
 Let $k_{0},m_{0}$ be such that $\alpha_{k_{0},\ell_{0},m_{0} }\neq0$.   From $(**)$ above we get:
 $$x^s\sum_{i=0}^p a_{i}e^i +\sum_{k,\ell,m} a_{k,\ell,m}e^ku(e-\ell(n+1))x^\ell   y^m =\sum_{k,\ell,m} a_{k,\ell,m}e^kx^{\ell+1} y^{m+1}.$$
 Using again $(7-1-2)$ we obtain:
 $$ \sum_{i=0}^p a_{i}(e-(n+1)s)^i +\sum_{k,\ell,m} a_{k,\ell,m}e^ku(e-\ell(n+1))x^\ell   y^m =\sum_{k,\ell,m} a_{k,\ell,m}e^kx^{\ell+1} y^{m+1}.$$
 The left hand side of the preceding equality does not contain the monomial $e^{k_{0}}x^  {\ell_{0}+1}y^{m_{0}+1}$, whereas the right hand side does. As the elements $e^kx^{\ell}y^m$ are a basis over ${\bf R}$ (Corollary \ref {cor.smith-basis}), we obtain a contradiction. Therefore  $\alpha=0$, and hence $x^s\sum_{i=0}^p a_{i}e^i=0$,  and again from Corollary \ref {cor.smith-basis}, we obtain that $a_{i}=0$ for all $i$. This proves that $\ker \varphi=\{0\}$. The proof for $\psi$ is similar.

 \end{proof}
 
 %%%%%%%%%%%%%%%%
 \begin{cor}\label{cor.U-basis} Every element $\tilde{u}$ in $U({\bf R},u,n)$ can be written in a unique way under the form
 $$\tilde{u}=\sum_{{\ell>0,k\geq 0}}\alpha_{k,\ell}\tilde{y}^{\ell}\tilde{e}^k+\sum_{m\geq 0,r\geq 0}\beta_{m,r}\tilde{x}^m \tilde{e}^r$$
 with $\alpha_{k,\ell}, \beta_{m,r}\in {\bf R}$.
 \end{cor}
 \begin{proof}
 From Corollary \ref{cor.smith-basis} and Remark \ref{rem-U=S/} we know that any element in $U({\bf R},u,n)$ can be written (in a non unique way) as a linear combination, with coefficients in ${\bf R}$, of the elements $ \tilde{x}^i\tilde{y}^j\tilde{e}^k$. 
 
 \noindent  Suppose that $i\geq j$. Then we have:
$$ \tilde{x}^i\tilde{y}^j\tilde{e}^k=  \tilde{x}^{i-j} \tilde{x}^j\tilde{y}^j\tilde{e}^k.$$
As $\tilde{y}\tilde{x}=u(\tilde{e}+n+1 )$ and $\tilde{x}\tilde{y}=u(\tilde{e})$, we see that $\tilde{x}^j\tilde{y}^j=Q_{j}(\tilde{e})$, where $Q_{j}$ is a polynomial with coefficients in ${\bf R}$. Therefore $\tilde{x}^i\tilde{y}^j\tilde{e}^k= \sum_{\ell} \gamma_{\ell}\tilde{x}^{i-j}\tilde{e}^{\ell}$, with $\gamma_{\ell}\in {\bf R}$. Similarly one can prove that if $i<j$, we have $\tilde{x}^i\tilde{y}^j\tilde{e}^k= \sum_{\ell} \delta_{\ell}\tilde{y}^{j-i}\tilde{e}^{\ell}$, with $\delta_{{\ell}}\in {\bf R}$. This shows that any element $\tilde{u}$ in $U({\bf R},u,n)$ can be written under the expected form.   

  \noindent Suppose now that:
 $$ \sum_{{\ell>0,k\geq 0}}\alpha_{k,\ell}\tilde{y}^{\ell}\tilde{e}^k+\sum_{m\geq 0,r\geq 0}\beta_{m,r}\tilde{x}^m \tilde{e}^r=0.$$
 Then, as $\tilde{y}^{\ell}\tilde{e}^k\in U({\bf R},u,n)^\ell$ and $\tilde{x}^m \tilde{e}^r\in U({\bf R},u,n)^m$, we deduce from $(7-2-1)$ that 
 \begin{align*}
 \forall  \ell>0 , \, \sum_{k}\alpha_{k,\ell}\tilde{y}^{\ell}\tilde{e}^k=0, \qquad  \forall  m \geq0, \,  \sum_{r}\beta_{m,r}\tilde{x}^m \tilde{e}^r=0. 
 \end{align*}
  Then from  Proposition \ref{prop.U-contient-pol}, we deduce that $\alpha_{k,\ell} =0$ and $\beta_{m,r} =0$.

 \end{proof}

 %%%%%%%%%%%%%%%%%%%%%%%%%%%%%%%%%%%%%%%%%%%%%%%%%%%%%%%
 \subsection{Generators and relations   for ${\cal T}_{0}[X,Y]={\bf D}(V^+)^{G'}$} \hfill
 \vskip 5pt
 
 \noindent  We know from Corollary \ref{cor-T0-poly-E} that any element $H\in {\cal T}_{0}$ can be written uniquely in  the form $H=u_{H}(E)$, where $u_{H}\in  {\cal   Z}( {\cal   T})[t]$. In particular $XY=u_{XY}(E)$. As the polynomial $u_{XY}$ will play an important role in the Theorem below, let us emphasize the connection between $u_{XY}$ and the Bernstein polynomial $b_{Y}$.
 First remark that $b_{Y}=b_{XY} $. Moreover  define $b_{Y}^{\rho}({\bf r})=b_{Y}({\bf r}-\rho)$. We know from section 5.2 that $b_{Y}^{\rho}$ is $S_{n+1}$-invariant and is the image under the  the Harish-Chandra isomorphism $\gamma $ of  the the $G$-invariant differential operator $XY$. In fact as $XY= D_{0}$ (see $(5-2-9)$, we have already compute $b_{Y}^{\rho}$ (see $(5-2-17)$): 
 $$b_{Y}^{\rho}({\bf r})=\prod_{i=0}^{n}(r_{i}+\frac{d}{4}n )$$From Corollary \ref{cor-decomp-pol-sym} this polynomial  can be written uniquely in the form 
 $$b_{Y}^{\rho}({\bf r})=\sum_{j} \beta_{j}({\bf   r})(r_{0}+r_{1}+\dots+r_{n})^{j}$$
 where $\beta_{j}\in {\cal Z}(\cal T)$. Then from section 5.2 we obtain:
 
 %%%%%%%%%%%%%%%%%
 \begin{prop}\label{prop-relation-u/b} 
 
 Keeping the notations above, we have
 $$u_{XY}(t)=\sum_{j}\gamma^{-1}(\beta_{j})t^{j}$$
 
 \end{prop}
 
 \vskip 5pt
   \noindent Let us now state the main result of this section:
%%%%%%%%%%%%%%%%
 \begin{theorem}\label{th.generateurs-relations-T_{0}} 

 \noindent  The mapping 
 $$\begin{array}{ccc}
 \tilde{x} \longmapsto X,\qquad \tilde{y}  \longmapsto Y,\qquad\tilde{e}   \longmapsto E\\
 \end{array} $$
 extends uniquely to an isomorphism of ${\cal Z}({\cal T})$-algebras between $ U({\cal Z}({\cal T}),u_{XY},n)$ and ${\cal T}_{0}[X,Y]$.
 \end{theorem}
 
 \begin{proof}
As  
$$[  E,  X]=(n+1)X, \quad [  E,Y]=-(n+1)Y, \quad   XY=u_{XY}(E), \quad  YX=u_{XY}( E+n+1),$$
and as ${\cal T}_{0}={\cal Z}({\cal T})[E]\simeq {\cal Z}({\cal T})[t]$, 
we know from the universal property of  $ U({\cal Z}({\cal T}),u_{XY},n)$, that the mapping 
$$\begin{array}{ccc}
 \tilde{x} \longmapsto X,\qquad \tilde{y}  \longmapsto Y,\qquad\tilde{e}   \longmapsto E\\
 \end{array} $$
 extends uniquely to a surjective morphism of ${\cal Z}({\cal T})$-algebras:
 $$\varphi:  U({\cal Z}({\cal T}),u_{XY},n)   \longrightarrow   {\cal T}_{0}[X,Y].$$
   Let  $\tilde{u}=\sum_{{\ell>0,k\geq 0}}\alpha_{k,\ell}\tilde{y}^{\ell}\tilde{e}^k+\sum_{m\geq 0,r\geq 0}\beta_{m,r}\tilde{x}^m \tilde{e}^r \in \ker \varphi $ (with $\alpha_{k,\ell},\beta_{m,r}\in {\cal Z}({\cal T})$, see Corollary \ref{cor.U-basis}). We have
 $$\begin{array}{rcl}
 \varphi(\tilde{u})&=&\displaystyle{\varphi(\sum_{{\ell>0,k\geq 0}}\alpha_{k,\ell}\tilde{y}^{\ell}\tilde{e}^k+\sum_{m\geq 0,r\geq 0}\beta_{m,r}\tilde{x}^m \tilde{e}^r)}\\
   &=&\displaystyle{\sum_{{\ell>0,k\geq 0}}\alpha_{k,\ell} Y^{\ell} E^k+\sum_{m\geq 0,r\geq 0}\beta_{m,r}X^m E^r}\\
   &=&0
 \end{array}$$
 Then Corollary \ref{cor-T0-poly-E} implies that $\alpha_{k,\ell}=\beta_{m,r}=0$, hence $\ker \varphi=\{0\}$.

 \end{proof}
 %%%%%%%%%%%%%%%%%%%%%%%%%%%%%%%%%%%%%%%%%%%%%%%%%%%%%%%%%%%%%%%%%%%%%%%%%%%%%%%%%%%%%%%%%%%%%%%%


\begin{thebibliography}{amsart}
% \bibitem[  ] {  }  { },      | {\it  },  
 %             \vskip 5pt 
  
 \vskip 5pt 
 \bibitem[B-R ] {Bopp-rub-complexe } N. B{\scriptsize OPP }, H. R{\scriptsize UBENTHALER }    | {\it Fonction z\^eta associ\'ee \`a la s\'erie principale sph\'erique de
certains espaces sym\'etriques}, Ann. Sci. Ecole Norm. Sup.   26 (1993) ${\rm n}^0$6, 701-745.
              \vskip 5pt
\bibitem[D] {Dixmier  } J.  D{\scriptsize IXMIER}  | {\it  Alg\`ebres
Enveloppantes},  Gauthiers-Villars (1974),Paris.

              \vskip 5pt 
              
\bibitem[F-K] {Far-Kor-art } J.  F{\scriptsize ARAUT}, A.  K{\scriptsize ORANYI}  | {\it Function Spaces and Reproducing Kernels on Bounded Symmetric Domains}, J. Funct. Anal. 88 (1990) ${\rm n}^{\circ}$1, 64-89.

              \vskip 5pt 
              
                             
 \bibitem[H-1] {Howe-Transcending} R. H{\scriptsize OWE},    | {\it Transcending Classical Invariant Theory}, Journal of the Amer. math. Soc. Volume 2, Number 3
 (1989), 535-552.

              \vskip 5pt 
                            
 \bibitem[H-2] {Howe-role-heisenberg} R. H{\scriptsize OWE},    | {\it On the Role of the Heisenberg Group in Harmonic Analyssis}, Bull. of the Americ. Math. Soc. Volume 3, Number 2 (1980), 821-843
              \vskip 5pt 
                            
 \bibitem[H-3] {Howe-dualpairs-physics} R. H{\scriptsize OWE},    | {\it Dual Pairs in Physics:
Harmonic Oscillators, Photons, Electrons, and Singletons}, Lectures in Applied Mathematics
21 (1985), 179-207.

              \vskip 5pt 
              \bibitem[H-S] {Heck-Schlich } G.  H{\scriptsize ECKMANN}, H.  S{\scriptsize CHLICHTKRULL}  | {\it Harmonic Analysis ans Special Functions on Symmetric Spaces}, Perspectives in Mathematics Vol. 14, Academic  Press (1994).

              \vskip 5pt 

                  
\bibitem[H-U] {Howe-Umeda} R. H{\scriptsize OWE},  T. U{\scriptsize MEDA}  | {\it The Capelli identity, the double commutant theorem, and multiplicity-free actions}, Math. Ann. 290
(1991), ${\rm n}^\circ 3$, 565-619.

              \vskip 5pt 
\bibitem[I] {Igusa} J. I{\scriptsize GUSA}   | {\it On Lie Algebras
Generated by Two Differential Operators},  Manifolds and Lie Groups , Progr. Math. 14 (1981)
Birkh\"auser, Boston, Mass., 187-195.

              \vskip 5pt 
 \bibitem[Ki] {Kimura} T. K{\scriptsize IMURA}   | {\it  Introduction to Prehomogeneous Vector Spaces},  Transl. Math. Monogr., vol 215, American Mathematical Society, Providence, RI, 2003.
              \vskip 5pt 
              \bibitem[Ko] {Koecher}M. K{\scriptsize   OECHER}   | {\it Imbedding of Jordan algebras into Lie algebras I},  Americ. J. Math.,    89 (1967),787-815.
\vskip 5pt


  \bibitem[K-S] {kos-sah} B. K{\scriptsize OSTANT}, S. S{\scriptsize AHI}  | {\it The Capelli identity, Tube Domains and the Generalized Laplace Transform},  Adv. Math. 87
(1991), 71-92.

              \vskip 5pt 
              
 \bibitem[L-M-O] {leroy-al} A. L{\scriptsize EROY}, J. {\scriptsize MATCZUK} and  J. O{\scriptsize KNINSKI}  | {\it On the Gelfand-Kirillov dimension of normal localizations and twisted
polynomial rings},   in Perspectives in Ring Theory, F. van Oystaeyen and L. Le Bruyn, Editors,
Kluwer Academic Publishers (1988), 205-214.
 
              \vskip 5pt 

 \bibitem[MC-R] {macConnell-rob} J.C. McC{\scriptsize  ONNEL} and J.C. R{\scriptsize OBSON}    | {\it  Non Commutative Noetherian Rings}, Wiley-interscience, Chichester, 1987.
\vskip 5pt 

 \bibitem[M-R-S] {M-R-S} I. M{\scriptsize ULLER}, H. R{\scriptsize UBENTHALER}  and G. S{\scriptsize CHIFFMANN}  | {\it  Structure des espaces pr\'ehomog\`enes
associ\'es \`a certaines alg\`ebres de Lie gradu\'ees},  Math. Ann. 274 (1986), 95-123.
\vskip 5pt 
 \bibitem[No] {  Nomura}  T. N{\scriptsize OMURA},      | {\it  Algebraically independent generators of invariant differential operators on a symmetric cone},  J. reine angew. Math. 400 (1989), 122-133.
           \vskip 5pt 

 \bibitem[Ra] {Rais} M. R{\scriptsize A\" IS}  | {\it  Distributions
homog\`enes sur des espaces de matrices}, Bull. Soc. Math. France, M\'emoire 30 (1972)

\vskip 5pt 

 \bibitem[Ra-S] {rallis-schif} S. R{\scriptsize ALLIS} and G. S{\scriptsize CHIFFMANN}    | {\it  Weil Representation. I Intertwining distributions and discrete
spectrum}, Mem. Amer. Math. Soc. 25 (1980) ${\rm n}^{\circ}231$.

\vskip 5pt 
 \bibitem[Ru-1] {rub-note-PV} H. R{\scriptsize UBENTHALER}      | {\it   Espaces vectoriel pr\'ehomog\`enes, sous-groupes paraboliques et ${\go {sl}}_{2}-triplets$}, C.R. Acad. Sci. Paris S\'er. A-B 290 (1980), $n^{\circ}$3, A127-A129. 
\vskip 5pt 
 \bibitem[Ru-2]{rub-these-etat}  H.    R{\scriptsize  UBENTHALER} |  {\it Espaces pr\'ehomog\`enes de
type parabolique}, Th\`ese d'Etat, Universit\'e de Strasbourg   (1982).
              \vskip 5pt

 
 \bibitem[Ru-3] {rub-bouquin-PV} H. R{\scriptsize UBENTHALER}      | {\it   Alg\`ebres de Lie et espaces pr\'ehomog\`enes}, in: Travaux en cours, Hermann, Paris, 1992.\vskip 5pt 


 \bibitem[R-S-1] {rub-schiff-1} H. R{\scriptsize UBENTHALER} and G. S{\scriptsize CHIFFMANN}    | {\it  Op\'erateurs diff\'erentiels de Shimura et espaces
 pr\'ehomog\`enes}, Invent. Math.   90 (1989), 409-442.

\vskip 5pt 

 \bibitem[R-S-2] {rub-schiff-2} H. R{\scriptsize UBENTHALER} and G. S{\scriptsize CHIFFMANN}    | {\it    $SL_2$-triplet associ\'e   \`a un polyn\^ome homog\`ene},  J. reine angew. math. 408 (1990), 136-158.   


\vskip 5pt 

 \bibitem[Sh] {shale} D. S{\scriptsize HALE}    | {\it  Linear symmetries of free boson fields},   Trans. Amer. Math. Soc. 103 (1962), 149-167.

\vskip 5pt 
 
\bibitem[S-K]{Sa-Ki}M. S{\scriptsize  ATO
} -- T. K{\scriptsize  IMURA}  | {\it A classification of irreducible
prehomogeneous
 vector spaces and their relative invariants}, Nagoya Math. J.  65 (1977),  1--155.
 \vskip 5pt
  \bibitem[Sm] {Smith} S. P. S{\scriptsize MITH}    | {A class of algebras similar to the enveloping algebra of ${\go s}{\go l}(2)$},   Trans. Amer. Math. Soc, 322, 1 (1990), 285-314.
\vskip 5pt 

 

 \bibitem[Ter] {terras} A. T{\scriptsize ERRAS}    |  {\it   Harmonic analysis
on symmetric spaces and applications I, II},   Springer (1988).

\vskip 5pt 

                           \bibitem[Ti]{Tits}  J.  T{\scriptsize  ITS} |  {Une classe d'alg\`ebres de Lie en relation avec les alg\`ebres de Jordan}, Indag. Math. 24, (1962), 530-535.
              \vskip 5pt


 \bibitem[Up] {upmeier} H. U{\scriptsize PMEIER}    | {\it   Jordan Algebras
and Harmonic Analysis on Symmetric Spaces},  Amer. J. Math. 108 (1986), 1-25.
\vskip 5pt 

 \bibitem[Wa] {wallach} N. W{\scriptsize ALLACH}    | {\it    Polynomial Differential Operators Associated with Hermitian Symmetric Spaces}, Representation Theory of Lie Groups and Lie Algebras, World Sci. Publishing, River Edge, NJ (1992), 76-94.


 
  \vskip 5pt 

 \bibitem[We] {weil} A. W{\scriptsize EIL}    | {\it Sur certains groupes d'op\'erateurs unitaires   }, Acta Math. 111 (1964), 143-211.
 
  

\vskip 5pt 

 \bibitem[Y] {yan} Z. Y{\scriptsize AN}    | {\it    Invariant Differential
Operators and Holomorphic Function Spaces},  J. Lie Theory 10 (2000) $n^\circ 1$, 1-31. 

\vskip 5pt 

 \bibitem[Z] {zhang-dimGK} J.J. Z{\scriptsize HANG}    | {\it    A note on GK
dimension of skew polynomial extensions},  Proc. Amer. Math. Soc. 125 (1997), no. 2, 363-373


\vskip 5pt 









 \end{thebibliography}
\end{document}